\documentclass{amsart}
\usepackage{amssymb}
\usepackage{mathrsfs}
\usepackage[backref,colorlinks=true]{hyperref}
\usepackage{tikz}
\usetikzlibrary{cd}
\usepackage{calligra}

\newcommand{\bk}{\Bbbk}
\newcommand{\Z}{\mathbb{Z}}
\newcommand{\C}{\mathbb{C}}
\newcommand{\Gm}{{\mathbb{G}_{\mathrm{m}}}}

\newcommand{\PCoh}{\mathsf{PCoh}}
\newcommand{\fC}{\mathfrak{C}}
\newcommand{\D}{\mathbb{D}}
\newcommand{\IC}{\mathcal{IC}}

\newcommand{\fg}{\mathfrak{g}}
\newcommand{\fu}{\mathfrak{u}}
\newcommand{\fb}{\mathfrak{b}}
\newcommand{\su}{\mathsf{u}}
\newcommand{\uenv}{\mathcal{U}}
\newcommand{\unip}{\mathrm{unip}}

\DeclareMathOperator{\ad}{ad}
\DeclareMathOperator{\Ad}{Ad}
\DeclareMathOperator{\Int}{Int}
\newcommand{\irr}{\mathsf{L}}

\newcommand{\tcN}{{\widetilde{\mathcal{N}}}}
\newcommand{\cN}{\mathcal{N}}

\newcommand{\rd}{{\mathsf{d}}}
\newcommand{\red}{{\mathrm{red}}}

\newcommand{\bX}{\mathbf{X}}
\newcommand{\aff}{{\mathrm{aff}}}

\newcommand{\cF}{\mathcal{F}}
\newcommand{\cG}{\mathcal{G}}
\newcommand{\cH}{\mathcal{H}}

\newcommand{\cL}{\mathcal{L}}
\newcommand{\cO}{\mathcal{O}}
\newcommand{\cV}{\mathcal{V}}

\newcommand{\Rep}{\mathsf{Rep}}
\newcommand{\Coh}{\mathsf{Coh}}
\newcommand{\Cohmix}{\mathsf{Coh}^{G \times \mathbb{G_{\mathrm{m}}}}}
\newcommand{\Db}{D^{\mathrm{b}}}

\DeclareMathOperator{\Hom}{Hom}
\DeclareMathOperator{\Ext}{Ext}
\DeclareMathOperator{\Sym}{Sym}
\DeclareMathOperator{\Ind}{Ind}
\DeclareMathOperator{\shHom}{\mathscr{H}\text{\kern -3pt {\calligra\large om}}\,}
\DeclareMathOperator{\codim}{codim}
\DeclareMathOperator{\im}{im}
\newcommand{\Irr}{\mathrm{Irr}}
\DeclareMathOperator{\Spec}{Spec}

\newcommand{\ol}[1]{\overline{#1}}
\newcommand{\id}{\mathrm{id}}
\newcommand{\simto}{\xrightarrow{\sim}}
\newcommand{\la}{\langle}
\newcommand{\ra}{\rangle}

\makeatletter
\def\@secnumfont{\bfseries}  
\makeatother

\newenvironment{psmallmatrix}
  {\left(\begin{smallmatrix}}
  {\end{smallmatrix}\right)}

\numberwithin{equation}{section}
\newtheorem{thm}{Theorem}[section]
\newtheorem{lem}[thm]{Lemma}
\newtheorem{prop}[thm]{Proposition}
\newtheorem{cor}[thm]{Corollary}

\theoremstyle{definition}

\theoremstyle{remark}
\newtheorem{rmk}[thm]{Remark}

\title[Calculations with graded perverse-coherent sheaves]{Calculations with graded\\ perverse-coherent sheaves}

 \author{Pramod N. Achar}
 \address{Department of Mathematics\\
   Louisiana State University\\
   Baton Rouge, LA 70803\\
   U.S.A.}
 \email{pramod@math.lsu.edu}

 \author{William D. Hardesty}
 \address{Department of Mathematics\\
   Louisiana State University\\
   Baton Rouge, LA 70803\\
   U.S.A.}
 \email{whardesty@lsu.edu}
 
 \thanks{P.A. was supported by NSF Grant No.~DMS-1500890.}

\begin{document}
\begin{abstract}
In this paper, we carry out several computations involving graded (or $\Gm$-equivariant) perverse-coherent sheaves on the nilpotent cone of a reductive group in good characteristic.  In the first part of the paper, we compute the weight of the $\Gm$-action on certain normalized (or ``canonical'') simple objects, confirming an old prediction of Ostrik.  In the second part of the paper, we explicitly describe all simple perverse coherent sheaves for $G = PGL_3$, in every characteristic other
than $2$ or $3$. Applications include an explicit description of the cohomology of tilting modules for the corresponding quantum group, 
as well as a proof that $\PCoh^{\Gm}(\cN)$ never admits a positive grading when the characteristic of the field is greater than~$3$. 
\end{abstract}

\maketitle

\section{Introduction}
\label{sec:intro}

Let $G$ be a reductive algebraic group over an algebraically closed field $\bk$, and let $\cN$ be its nilpotent cone.  The derived category $\Db\Coh^G(\cN)$ of $G$-equivariant coherent sheaves on $\cN$ admits a remarkable $t$-structure whose heart is known as the category of \emph{perverse-coherent sheaves}, and is denoted by $\PCoh(\cN)$.  This category has some features in common with classical (constructible) perverse sheaves: most importantly, every object has finite length, and the simple objects are produced by an ``intermediate-extension'' (or ``IC'') construction, starting from a pair $(C,\cV)$, where $C \subset \cN$ is a nilpotent orbit, and $\cV$ is an irreducible $(G \times \Gm)$-equivariant vector bundle on $C$. For applications of perverse-coherent sheaves to representation theory, see~\cite{ahr, arider, ab:pcs, bezru, bezru-psaf}.

Now let the multiplicative group $\Gm$ act on $\cN$ by $z \cdot x = z^{-2}x$.  One can then consider the derived category $\Db\Cohmix(\cN)$ of $(G \times \Gm)$-equivariant coherent sheaves, along with the full subcategory of ``graded'' perverse-coherent sheaves, denoted by $\PCoh^\Gm(\cN)$.  In this paper, we carry out several computations in this category.

In the first part of the paper, we study lifts of simple objects from $\PCoh(\cN)$ to $\PCoh^\Gm(\cN)$.  These lifts are not unique, of course: any lift can be twisted by a character of $\Gm$ to obtain another lift. But in the context of the \emph{Lusztig--Vogan bijection} (see Section~\ref{sec:gradedlv} for an explanation), each simple object in $\PCoh(\cN)$ admits a \emph{canonical} lift to $\PCoh^\Gm(\cN)$.  In Theorem~\ref{thm:graded-lv}, we compute the weight of the $\Gm$-action on the canonical lift of $\IC(C,\cV)$: it turns out to be $-\frac{1}{2}\codim \ol{C}$.  (In characteristic $0$, this result was predicted about twenty years ago by Ostrik~\cite{ostrik-Kth}.)

In the second part of the paper, we explicitly compute all simple perverse-coherent sheaves for the group $G = PGL_3$, for all characteristics other than $2$ or $3$.  For some context, we remark that at the moment, there is no known algorithm for computing simple perverse-coherent sheaves in general, even in characteristic $0$.  Here are the cases in which $\IC(C,\cV)$ was previously known:
\begin{itemize}
\item If $C$ is the zero nilpotent orbit, the problem is trivial.
\item If $\cV = \cO_C$ is the trivial vector bundle on $C$, then by~\cite[Remark~11]{bezru}, $\IC(C,\cO_C)$ is (a shift of) the push-forward of the structure sheaf of the normalization of $\overline{C}$.
\item For $G = SL_2$, there is one simple perverse-coherent sheaf (up to grading shift) that does not fall into one of the cases above, corresponding to the nontrivial line bundle on the principal nilpotent orbit.  This object is described in~\cite{achar} (but was known to experts before; cf.~\cite[Remark~11]{bezru}).
\end{itemize}
The main result of the second part of the paper (Theorem~\ref{thm:main}) adds a number of new cases to this list of examples. Here are some features of this computation.
\begin{itemize}
\item We exhibit the first explicit examples of $\IC$'s that are not concentrated in single cohomological degree.
\item In characteristic $0$, by~\cite{bezru}, our results give an explicit computation of the cohomology of the small quantum group $\su_\zeta(\mathfrak{sl}_3)$ with coefficients in a tilting module $T$.  Remarkably, we find that if $T$ is nontrivial, $H^i(\su_\zeta(\mathfrak{sl}_3), T)$ is irreducible as a $PGL_3$-representation (whenever it is nonzero).
\item Our examples show that if the characteristic of $\bk$ is larger than $3$, then $\PCoh^{\Gm}(\cN)$ does not admit a positive grading.
\end{itemize}

Here are some further remarks on the ``positive grading'' phenomenon (see Section~\ref{sec:appl} for the definition).  For any group $G$ in characteristic~$0$, $\PCoh^\Gm(\cN)$ admits a positive grading (cf.~\cite{bezru-psaf}). Furthermore, the calculations in~\cite{achar} show that for $G = SL_2$, $\PCoh^{\Gm}(\cN)$  admits a positive grading in all characteristics other than $2$.  The latter might lead one to hope for a general positivity theorem for $\PCoh^\Gm(\cN)$, but the results of this paper provide counterexamples.  This phenomenon is likely related to other positivity questions arising in geometric representation theory, such as those considered in~\cite{ar:mpsfv3,lw}.

Finally, we expect that the results in this paper will be useful for computing the cohomology of tilting modules for $PGL_3$ in positive characteristic.  (An ``abstract'' solution to this problem appears in~\cite[Proposition~9.1]{ahr}.)  We hope to return to this question some time in the future.

The paper is organized as follows. Sections~\ref{sec:prelim} and~\ref{sec:pcoh} contain general background and lemmas on nilpotent orbits, canonical sheaves, and perverse-coherent sheaves.  Section~\ref{sec:gradedlv} contains the first main result of the paper.

From Section~\ref{sec:notation} on, we restrict our attention to the group $G = PGL_3$.  In Section~\ref{sec:geometry}, we study a particular resolution of the middle nilpotent orbit.  The second main theorem appears in Section~\ref{sec:mainthm}. Finally, the applications to quantum group cohomology and to positivity questions appear in Section~\ref{sec:appl}.

\subsection*{Acknowledgments}

We are grateful to Simon Riche and to an anonymous referee for a very careful reading of this paper.

\section{Preliminaries on nilpotent orbits}\label{sec:prelim}

\subsection{Notation and conventions}

Let $\bk$ be an algebraically closed field.  For a graded $\bk$-vector space $V = \bigoplus_{m \in \Z} V_m$, we define $V\la n\ra$ to be the graded $\bk$-vector space given by $(V\la n\ra)_m = V_{m+n}$.  It is sometimes convenient to think of $V$ as a $\Gm$-representation, where $V_m$ is the $m$-weight space.  In this language, we have $V\la n\ra \cong V \otimes \bk_{-n}$, where $\bk_{-n}$ is the $1$-dimensional $\Gm$-representation of weight $-n$.

Let $G$ be a connected, reductive algebraic group over $\bk$, and let $\fg$ be its Lie algebra.  We assume that $G$ and $\bk$ satisfy the following conditions:
\begin{enumerate}
\item[(H1)] There exists a separable isogeny $\tilde G \to G$, where $\tilde G$ is a reductive group whose derived subgroup is simply connected.
\item[(H2)] The characteristic of $\bk$ is good for $G$.
\item[(H3)] There exists a nondegenerate $G$-invariant bilinear form on $\fg$.
\end{enumerate}
These conditions are very close to those in~\cite[\S2.9]{jantzen-nilp}, except that in \textit{loc.~cit.}, condition~(H1) is replaced by the stronger condition that $G$ itself have a simply-connected derived subgroup.  Below, we will invoke some results from~\cite{jantzen-nilp} that are stated under the assumptions of~\cite[\S2.9]{jantzen-nilp}.  Using~\cite[Proposition~2.7(b)]{jantzen-nilp}, one can easily check that these results remain valid under our weaker assumptions.

Choose a nondegenerate $G$-invariant bilinear form
\begin{equation}\label{eqn:bilin}
B_\fg: \fg \times \fg \to \bk
\end{equation}
as in~(H3). Choose a maximal torus and a Borel subgroup $T \subset B \subset G$.  Let $\bX$ be the weight lattice of $T$. We declare the roots corresponding to $B$ to be the negative roots, and then let $\bX^+ \subset \bX$ be the corresponding set of dominant weights.

Let $\cN$ be the nilpotent cone of $G$.  As in~\S\ref{sec:intro}, we let $\Gm$ act on $\cN$ by $z \cdot x = z^{-2}x$.  This makes the coordinate ring $\bk[\cN]$ into a graded ring that is concentrated in even, nonnegative degrees.  For $\cF \in \Db\Cohmix(\cN)$, the notation $\cF\la n\ra$ is defined similarly.

Let $C \subset \cN$ be a nilpotent orbit.  The $\Gm$-action defined above preserves every nilpotent orbit, so it makes sense to consider $(G \times \Gm)$-equivariant coherent sheaves on $C$.  Choose a representative $x_C \in C$, and let $G^{x_C}$, resp.~$(G \times \Gm)^{x_C}$, be its stabilizer in $G$, resp.~$G \times \Gm$.  Equip $C$ with the reduced locally closed subscheme structure.  Our assumptions (H1)--(H3) imply that there are isomorphisms of varieties
\[
C \cong G/G^{x_C} \cong (G \times \Gm)/(G \times \Gm)^{x_C}.
\]
(See~\cite[\S2.2 and \S2.9]{jantzen-nilp} for the former; the latter is similar.)  We therefore have equivalences of categories
\begin{equation}\label{eqn:indequiv-orbit}
\Coh^G(C) \cong \Rep(G^{x_C}),
\qquad
\Cohmix(C) \cong \Rep((G \times \Gm)^{x_C}).
\end{equation}
Next, let $\fg^{x_C} = \ker (\ad(x_C))$. As explained in~\cite[\S2.2 and \S2.9]{jantzen-nilp}, our assumptions imply that the tangent space $T_{x_C}C$ to $C$ at $x_C$ can be identified as
\[
T_{x_C}C \cong [x_C,\fg] \subset \fg.
\]
Using $B_\fg$, one can see that the cotangent space $T_{x_C}^*C \cong [x_C,\fg]^*$ is isomorphic as a $G^{x_C}$-representation to $\fg/\fg^{x_C}$.  However, this is \emph{not} an isomorphism of $(G \times \Gm)^{x_C}$-representations.  Since $\Gm$ acts on $\fg$ (and hence $\fg/\fg^{x_C}$) with weight $-2$, it acts on $\fg^*$ (and hence on $[x_C,\fg]^*$) with weight $2$.  We therefore have  
\begin{equation}\label{eqn:gxc-cotangent}
T_{x_C}^*C\la 4\ra \cong \fg/\fg^{x_C}.
\end{equation}

The following fact is well known in characteristic $0$ (cf.~\cite[\S1.4]{cm:nosla}).  The same reasoning goes through in positive characteristic as well.  For completeness, we include the proof.

\begin{lem}\label{lem:cotangent-sympl}
The cotangent space $T_{x_C}^*C$ admits a nondegenerate $G^{x_C}$-invariant symplectic form.
\end{lem}
\begin{proof}
Define a new bilinear form
\[
\tilde\omega : \fg \times \fg \to \bk
\qquad\text{by}\qquad
\tilde\omega(u,v) = B_\fg(x_C,[u,v]).
\]
This form is $G^{x_C}$-invariant, and it has the property that $\tilde\omega(u,u) = 0$.  However, it is not nondegenerate: its radical consists of vectors $u \in \fg$ such that
\[
B_\fg(x_C,[u,v]) = B_\fg([x_C,u],v) = 0 \qquad\text{for all $v \in \fg$.}
\]
But since $B_\fg$ is nondegenerate, this happens if and only if $u \in \fg^{x_C}$.  We conclude that $\tilde\omega$ induces a nondegenerate $G^{x_C}$-invariant symplectic form on $\fg/\fg^{x_C}$. In view of~\eqref{eqn:gxc-cotangent} (ignoring the $\Gm$-action), we are done.
\end{proof}

\subsection{Associated cocharacters}

For each nilpotent orbit $C$, choose an \emph{associated cocharacter} $\phi_{x_C}: \Gm \to G$ in the sense of~\cite[Definition~5.3]{jantzen-nilp}. Decompose the adjoint representation into weights for $\phi_{x_C}$:
\[
\fg = \bigoplus_{k \in \Z} \fg(k) 
\qquad
\text{where $\fg(k) = \{ v \mid \Ad(\phi_{x_C}(t))(v) = t^kv \}$.}
\]
Condition~(H3) implies that
\begin{equation}\label{eqn:gk-pairing}
\dim \fg(k) = \dim \fg(-k).
\end{equation}

As part of the definition of an associated cocharacter, we have
\begin{equation}\label{eqn:assoc-wt2}
x_C \in \fg(2).
\end{equation}
This implies that the group $\phi_{x_C}(\Gm)$ normalizes $G^{x_C}$.  Consider the semidirect product $\Gm \ltimes G^{x_C}$, where $\Gm$ acts on $G^{x_C}$ by $z \cdot g = \Int_{\phi_{x_C}(z)}(g)$.  It is easy to see that there is an isomorphism
\begin{equation}\label{eqn:graded-centralizer}
\Gm \ltimes G^{x_C} \simto (G \times \Gm)^{x_C}
\qquad\text{given by}\qquad
(z,g) \mapsto (\phi_C(z)g, z).
\end{equation}
Let $G^{x_C}_\unip$ be the unipotent radical of $G^{x_C}$, and let
\begin{equation}\label{eqn:gxred-defn}
G^{x_C}_\red = \{ g \in G^{x_C} \mid \text{$\phi_{x_C}(z)g = g\phi_{x_C}(z)$ for all $z \in \Gm$} \}.
\end{equation}
The notation is justified by~\cite[Propositions~5.10 and~5.11]{jantzen-nilp}, which tell us that $G^{x_C}_\red$ is a reductive (but possibly disconnected) group, and that $G^{x_C} = G^{x_C}_\red \ltimes G^{x_C}_\unip$.

The group $G^{x_C}_\unip$ is also the unipotent radical of $\Gm \ltimes G^{x_C}$, so it acts trivially on any irreducible representation.  In other words, an irreducible representation of $\Gm \ltimes G^{x_C}$ is the same as an irreducible representation of the group
\[
\Gm \ltimes G^{x_C}_\red = \Gm \times G^{x_C}_\red.
\]
In particular, any irreducible $G^{x_C}_\red$-module can be regarded as an irreducible $\Gm \ltimes G^{x_C}$-module by making $\Gm$ act on it trivially.  This construction defines an embedding
\begin{equation}\label{eqn:irred-embed}
\Irr(G^{x_C}) \hookrightarrow \Irr(\Gm \ltimes G^{x_C}).
\end{equation}
Moreover, every irreducible $(\Gm \ltimes G^{x_C})$-module is obtained by applying some grading shift $\la n\ra$ to a representation in the image of this map.

It follows from~\eqref{eqn:assoc-wt2} that $\phi_{x_C}(\Gm)$ preserves $\fg^{x_C}$, and hence that the latter also decomposes into weight spaces.  According to~\cite[Proposition~5.8]{jantzen-nilp}, only nonnegative weights occur:
\begin{equation}\label{eqn:gxc-positive}
\fg^{x_C} = \bigoplus_{k \ge 0} \fg^{x_C}(k)
\qquad
\text{where $\fg^{x_C}(k) = \fg(k) \cap \fg^{x_C}$.}
\end{equation}

\begin{lem}
\label{lem:dimc-calc}
We have
\[
\sum_{k \ge 0} k \dim \fg^{x_C}(k) = \dim C.
\]
\end{lem}
\begin{proof}
We begin with the claim that for any $k \ge 0$, we have
\begin{equation}\label{eqn:dimc-claim}
\dim \fg(k) = \sum_{j \ge 0} \dim \fg^{x_C}(k + 2j).
\end{equation}
We prove this by downward induction on $k$.  If $k \gg 0$, then $\fg(k)$ and all the $\fg^{x_C}(k+2j)$ vanish, so~\eqref{eqn:dimc-claim} holds trivially.

Now suppose that~\eqref{eqn:dimc-claim} is known when $k$ is replaced by $k+2$.  Since $k \ge 0$, we have seen in~\eqref{eqn:gxc-positive} that $\fg^{x_C}(-k-2) = 0$.  Then, by~\cite[Lemma~5.7]{jantzen-nilp}, the map $\ad(x_C): \fg(k) \to \fg(k+2)$ is surjective, so
\begin{multline*}
\dim \fg(k) = \dim \fg^{x_C}(k) + \dim \fg(k+2)  \\
= \dim \fg^{x_C}(k) + \sum_{j \ge 0} \dim \fg^{x_C}(k+2+2j)
= \sum_{j\ge 0} \dim \fg^{x_C}(k+2j),
\end{multline*}
as desired.

In view of~\eqref{eqn:gk-pairing} and~\eqref{eqn:dimc-claim}, we have
\begin{multline}\label{eqn:dimc-claim2}
\dim \fg = \sum_{k \in \Z} \dim \fg(k) = \dim \fg(0) + 2\sum_{k \ge 1} \dim \fg(k) \\
= \sum_{j \ge 0} \dim \fg^{x_C}(2j)
+ 2 \sum_{k \ge 1} \sum_{j\ge 0} \dim \fg^{x_C}(k+2j).
\end{multline}
Let us rearrange this sum in the form $\dim \fg = \sum_{i \ge 0} c_i \dim \fg^{x_C}(i)$ for some integer coefficients $c_i$.  Examining~\eqref{eqn:dimc-claim}, it is easy to see that in fact we have $c_i = i+1$ for all $i$.  That is,
\[
\dim \fg = \sum_{i \ge 0} (i+1) \dim \fg^{x_C}(i).
\]
To conclude, we note that $\dim C = \dim \fg - \dim \fg^{x_C}$ is given by
\[
\dim C = \sum_{i\ge 0} (i+1) \dim \fg^{x_C}(i) - \sum_{i \ge 0} \dim \fg^{x_C}(i)
= \sum_{i\ge 0} i \dim \fg^{x_C}(i),
\]
as desired.
\end{proof}

\subsection{Canonical bundles of nilpotent orbits}

This subsection contains computations related to the $(G \times \Gm)$-equivariant structure of the canonical sheaf of a nilpotent orbit or its closure.

\begin{prop}\label{prop:orbit-canbundle}
Let $C \subset \cN$ be a nilpotent orbit. The canonical bundle $\omega_C$ is isomorphic as a $(G \times \Gm)$-equivariant coherent sheaf to $\cO_C\la -\dim C\ra$.
\end{prop}
This statement tells us that if we forget the $\Gm$-equivariance, the canonical bundle of any nilpotent orbit is trivial.
\begin{proof}
Via~\eqref{eqn:indequiv-orbit} and~\eqref{eqn:graded-centralizer}, we identify $\Cohmix(C)$ with $\Rep(\Gm \ltimes G^{x_C})$.  The canonical bundle is the top exterior power of the cotangent bundle, so we must show that as a $(\Gm \ltimes G^{x_C})$-representation, $\bigwedge^{\dim C} T_{x_C}^*C$ is isomorphic to $\bk\la -\dim C\ra$.

By Lemma~\ref{lem:cotangent-sympl}, the action of $G^{x_C}$ on $T_{x_C}^*C$ factors through the symplectic group $\mathrm{Sp}(T_{x_C}^*C)$.  It is easy to see that $\mathrm{Sp}(T_{x_C}^*C)$ acts trivially on $\bigwedge^{\dim C} T_{x_C}^*C$, so $G^{x_C}$ does as well.

On the other hand, the weight of the $\Gm$-action on $\bigwedge^{\dim C} T_{x_C}^*C$ is the sum of its weights (with multiplicities) on $T_{x_C}^*C$.  Using~\eqref{eqn:gxc-cotangent}, we have
\[
\sum(\text{weights of $[x_C,\fg]^*$}) = 
4 \dim C + \sum(\text{weights of $\fg$})
-\sum(\text{weights of $\fg^{x_C}$}).
\]
From~\eqref{eqn:graded-centralizer}, we see that $\Gm$ acts on $\fg(k)$ or on $\fg^{x_C}(k)$ with weight $k - 2$.  The $\Gm$-weight of $\bigwedge^{\dim C} T_{x_C}^*C$ is therefore given by
\begin{multline*}
4 \dim C + \sum_{k \in \Z} (k-2) \dim \fg(k) - \sum_{k \ge 0} (k-2) \dim \fg^{x_C}(k) 
\\
= 4 \dim C - 2 (\dim \fg - \dim \fg^{x_C}) + \sum_{k \in \Z} k \dim \fg(k) - \sum_{k \ge 0} k \dim \fg^{x_C}(k) = \dim C.
\end{multline*}
Here, we have used Lemma~\ref{lem:dimc-calc} and the fact (implied by~\eqref{eqn:gk-pairing}) that $\sum k \dim \fg(k) = 0$.
\end{proof}

\begin{prop}\label{prop:gorenstein-dc}
Let $C$ be a nilpotent orbit,  and let $a: \ol{C} \to \Spec\bk$ be the structure map.   Then $(a^!\cO_{\Spec\bk})|_C \cong \cO_C[\dim C]\la -\dim C\ra$.  If $\ol{C}$ is Gorenstein, then $a^!\cO_{\Spec\bk} \cong \cO_{\ol{C}}[\dim C]\la -\dim C\ra$.
\end{prop}
\begin{proof}
Let $a_0: C \to \Spec\bk$ be the structure map of the orbit.  Since $C$ is an open subset of $\ol{C}$, we have $(a^!\cO_{\Spec \bk})|_C \cong a_0^!\cO_{\Spec \bk}$.  Since $C$ is a smooth variety, it is well known that $a_0^!\cO_{\Spec \bk} \cong \omega_C[\dim C]$.  By Proposition~\ref{prop:orbit-canbundle}, we have $a_0^!\cO_{\Spec \bk} \cong \cO_C[\dim C]\la -\dim C\ra$.

Now, $a^!\cO_{\Spec \bk}$ is a dualizing complex for $\ol{C}$.  If $\ol{C}$ is Gorenstein, any dualizing complex is a shift of a line bundle, by, say,~\cite[Tag 0BFQ]{stacks-project}.  By the preceding paragraph, $(a^!\cO_{\Spec \bk})|_C$ is concentrated in cohomological degree $-\dim C$, so $a^!\cO_{\Spec \bk}$ must be of the form $\cL[\dim C]$ for some line bundle $\cL$.  The line bundle $\cL$ is also determined by its restriction to $C$: because the complement of $C$ in $\ol{C}$ has codimension at least $2$, we know by~\cite[Theorem~1.12]{hartshorne3} that the restriction functor $\Cohmix(\ol{C}) \to \Cohmix(C)$ is fully faithful on reflexive sheaves, and in particular on line bundles.  Since $\cL|_C \cong \cO_C\la -\dim C\ra \cong \cO_{\ol{C}}\la -\dim C\ra|_C$, we conclude that $\cL \cong \cO_{\ol{C}}\la -\dim C\ra$.
\end{proof}

\begin{cor}\label{cor:orbit-dc}
Let $C$ be a nilpotent orbit, and let $i: \ol{C} \hookrightarrow \cN$ be the inclusion map. Then $(i^!\cO_\cN)|_C \cong \cO_C[-\codim C]\la \codim C\ra$.  If $\ol{C}$ is Gorenstein, then $i^!\cO_\cN \cong \cO_{\ol{C}}[-\codim C]\la \codim C\ra$.
\end{cor}
\begin{proof}
Let $a_\cN: \cN \to \Spec \bk$ and $a: \ol{C} \to \Spec \bk$ be the structure maps.  The variety $\cN$ is Gorenstein by~\cite[Theorem~5.3.2]{bk}, so by Proposition~\ref{prop:gorenstein-dc} applied to the principal nilpotent orbit, we have
\[
i^!\cO_\cN \cong i^! a_\cN^!\cO_{\Spec \bk}[-\dim \cN] \la \dim \cN\ra
\cong a^!\cO_{\Spec \bk}[-\dim \cN] \la \dim \cN\ra.
\]
The claims follow by another application of Proposition~\ref{prop:gorenstein-dc}.
\end{proof}

\section{Perverse-coherent sheaves}\label{sec:pcoh}

\subsection{Andersen--Jantzen sheaves}

Let $\fu$ be the Lie algebra of the unipotent radical of $B$, and let $\tcN = G \times^B \fu$.  Let $\pi: \tcN \longrightarrow \cN$ be the Springer resolution.  Let $\Gm$ act on $\fu$ by $z \cdot x = z^{-2}z$.  Then there is an induced action of $\Gm$ on $\tcN$ that commutes with the $G$-action, and the map $\pi$ is $(G \times \Gm)$-equivariant.

For $\lambda \in \bX$, let $\bk_\lambda$ be the corresponding $1$-dimensional $B$-representation, and let $\cO_{G/B}(\lambda)$ be the corresponding line bundle on $G/B$.  We regard this as a $(G \times \Gm)$-equivariant line bundle by having $\Gm$ act trivially on $\bk_\lambda$.  Let $p: \tcN \to G/B$ be the projection map, and let $\cO_\tcN(\lambda) = p^*\cO_{G/B}(\lambda)$.  Finally, set
\[
A_\lambda = \pi_*\cO_\tcN(\lambda) \in \Db\Cohmix(\cN).
\]
This object is called an \emph{Andersen--Jantzen sheaf}.  This sheaf (or rather, complex of sheaves) can also be written down as a complex of $G$-representations:
\begin{equation}\label{eqn:aj-sheaf}
A_\lambda \cong R\Ind_B^G (\Sym(\fu^*) \otimes \bk_\lambda).
\end{equation}
In this language, $A_\lambda$ is a complex of modules over the ring
\begin{equation}\label{eqn:kn-ind}
\bk[\cN] \cong R\Ind_B^G (\Sym(\fu^*)).
\end{equation}
(For the vanishing of higher cohomology on the right-hand side of this isomorphism, see~\cite[Theorem~5.2.1]{bk}.)

In~\eqref{eqn:aj-sheaf} and~\eqref{eqn:kn-ind}, we can replace $B$ by another Borel subgroup, if we also replace $\bk_\lambda$ by the appropriate conjugate.  For example, let $B^+ \subset G$ be the Borel subgroup containing $T$ and opposite to $B$, and let $\fu^+$ be the Lie algebra of its unipotent radical.  Since $B^+$ is obtained from $B$ by conjugating by (a lift to $G$ of) the longest element of the Weyl group $w_0$, we have
\begin{equation}\label{eqn:aj-opp}
A_\lambda \cong R\Ind_{B^+}^G (\Sym((\fu^+)^*) \otimes \bk_{w_0\lambda}).
\end{equation}
Here, the right-hand side is a complex of modules over the ring
\begin{equation}\label{eqn:kn-opp}
\bk[\cN] \cong R\Ind_{B^+}^G (\Sym((\fu^+)^*)).
\end{equation}

\subsection{Serre--Grothendieck duality}

The \emph{Serre--Grothendieck duality functor} is defined to be the functor $\D_\cN: \Db\Cohmix(\cN) \to \Db\Cohmix(\cN)$ given by
\[
\D_\cN = R\shHom({-},\cO_\cN).
\]
This functor is an antiautoequivalence, and it satisfies $\D \circ \D \cong \id$.  Note that we are making a choice of normalization here: one could use another dualizing complex on $\cN$ instead of $\cO_\cN$.  This choice agrees with that in~\cite{achar-pcoh, achar, bezru}, but not with~\cite{bez:qes}.

More generally, for any nilpotent orbit $C$, we define $\D_{\ol{C}}: \Db\Cohmix(\ol{C}) \to \Db\Cohmix(\ol{C})$ by
\[
\D_{\ol{C}} = R\shHom({-},i_{\ol{C}}^! \cO_\cN),
\]
where $i_{\ol{C}}: \ol{C} \hookrightarrow \cN$ is the inclusion map.  With these conventions, we have
\begin{equation}\label{eqn:dual-olc}
i_{\ol{C}*} \circ \D_{\ol{C}} \cong \D_{\cN} \circ i_{\ol{C}*}.
\end{equation}
We also define $\D_\tcN: \Db\Cohmix(\tcN) \to \Db\Cohmix(\tcN)$ by $\D_\tcN = R\shHom({-}, \cO_{\tcN})$.  We again have $\pi_* \circ \D_\tcN \cong \D_\cN \circ \pi_*$.  As an immediate consequence, we have
\begin{equation}\label{eqn:aj-dual}
\D_\cN(A_\lambda) \cong A_{-\lambda}.
\end{equation}

\subsection{Perverse-coherent sheaves}

For $\lambda \in \bX^+$, let $\delta_\lambda = \min \{ \ell(w) \mid w\lambda \in -\bX^+ \}$, and then let
\begin{equation}\label{eqn:std-defn}
\Delta_\lambda = A_{w_0\lambda}\la \delta_\lambda\ra
\qquad\text{and}\qquad
\nabla_\lambda = A_\lambda\la -\delta_\lambda\ra.
\end{equation}
It follows from~\eqref{eqn:aj-dual} that
\begin{equation}\label{eqn:std-dual}
\D(\Delta_\lambda) \cong \nabla_{-w_0\lambda}.
\end{equation}

Recall that the \emph{perverse-coherent t-structure} on $\Db\Cohmix(\cN)$ is the pair of full subcategories $({}^pD^{\le 0}, {}^pD^{\ge 0})$ of $\Db\Cohmix(\cN)$ defined as follows:
\begin{align*}
{}^p D^{\le 0} &= 
\begin{array}{c}
\text{the subcategory generated under} \\
\text{extensions by $\{ \Delta_\lambda\la n\ra[k] : \lambda \in \bX^+, n \in \Z, k \ge 0 \}$,}
\end{array}
\\
{}^p D^{\ge 0} &= 
\begin{array}{c}
\text{the subcategory generated under} \\
\text{extensions by $\{ \nabla_\lambda\la n\ra[k] : \lambda \in \bX^+, n \in \Z, k \le 0 \}$.}
\end{array}
\end{align*}
Let $\PCoh^\Gm(\cN) = {}^pD^{\le 0} \cap {}^pD^{\ge 0}$ be the heart of this t-structure.  This is the category of $(G \times \Gm)$-equivariant perverse-coherent sheaves on $\cN$.  For background and general information on these objects, see~\cite{achar-pcoh, achar, ab:pcs, bez:qes}.  Some key features are as follows: $\PCoh^\Gm(\cN)$ is preserved by $\D_\cN$, and every object has finite length.  For each $\lambda \in \bX^+$, the objects $\Delta_\lambda$ and $\nabla_\lambda$ belong to $\PCoh^\Gm(\cN)$, and there is a canonical morphism $\Delta_\lambda \to \nabla_\lambda$.  Denote the image of this morphism by
\[
\fC_\lambda = \im(\Delta_\lambda \to \nabla_\lambda).
\]
This is a simple object, and up to grading shift, every simple object is of this form.  It follows from~\eqref{eqn:std-dual} that
\begin{equation}\label{eqn:ic-dual}
\D_\cN(\fC_\lambda) \cong \fC_{-w_0\lambda}.
\end{equation}

There is a second approach to classifying simple perverse-coherent sheaves, as follows. Given a nilpotent orbit $C \subset \cN$ and an irreducible $(G \times \Gm)$-equivariant vector bundle $\cV$ on $C$, there is (functorial) perverse-coherent sheaf
\[
\IC(C,\cV)
\]
that is characterized by the following properties: it is supported on $\ol{C}$; it satisfies
\[
\textstyle
\IC(C,\cV)|_C \cong \cV[-\frac{1}{2}\codim C];
\]
and it has no nonzero subobject or quotient supported on $\ol{C} \smallsetminus C$. It turns out that $\IC(C,\cV)$ is simple, and every simple perverse-coherent sheaf arises in this way.  In view of~\eqref{eqn:indequiv-orbit} and~\eqref{eqn:graded-centralizer}, we can replace the vector bundle $\cV$ by an irreducible $(\Gm \ltimes G^{x_C})$-representation.

If $V$ is a $G^{x_C}$-representation, regard it as a $(\Gm \ltimes G^{x_C})$-representation by having $\Gm$ act trivially (cf.~\eqref{eqn:irred-embed} and the discussion preceding it).  Then every irreducible $(\Gm \ltimes G^{x_C})$-representation is of the form $V\la n\ra$ for some irreducible $G^{x_C}$-representation $V$ and some integer $n$.  To summarize, we have a canonical bijection
\begin{align}\label{eqn:ic-classify}
\left\{
\begin{array}{c}
\text{simple perverse-} \\
\text{coherent sheaves}
\end{array}
\right\}
&\leftrightarrow
\left\{ (C,V) \Big|
\begin{array}{c}
\text{$C$ a nilpotent orbit, and $V$ an} \\
\text{irreducible $G^{x_C}$-representation}
\end{array}
\right\} \times \Z, \\
\IC(C,V)\la n\ra &\leftrightarrow ((C,V),n). \notag
\end{align}
We will compare the two classifications of simple objects in the next subsection.

\begin{lem}\label{lem:ic-dual-calc}
Let $C$ be a nilpotent orbit, and let $V$ be a $G^{x_C}$-representation.  For any $n \in \Z$, we have $\D(\IC(C,V)\la n\ra) \cong \IC(C, V^*)\la \codim C - n\ra$.
\end{lem}
\begin{proof}
Since $\D$ is an equivalence of categories, $\D(\IC(C,V)\la n\ra)$ is still a simple object of $\PCoh^\Gm(\cN)$. By~\eqref{eqn:dual-olc}, $\D(\IC(C,V)\la n\ra)$ is still supported on $\ol{C}$, so to compute it, it is enough to determine the vector bundle
\[
\textstyle
(\D(\IC(C,V)\la n\ra))|_C[\frac{1}{2} \codim C].
\]
By Corollary~\ref{cor:orbit-dc}, this is given by
\begin{multline*}
\textstyle
R\shHom(V[-\frac{1}{2}\codim C]\la n\ra, \cO_C[-\codim C]\la\codim C\ra)[\frac{1}{2} \codim C] \\
\cong V^*\la \codim C -n\ra,
\end{multline*}
as desired.
\end{proof}

\begin{rmk}\label{rmk:pcoh-0}
Let $C_0$ be the zero nilpotent orbit, and let $i_0: C_0 \hookrightarrow \cN$ be the inclusion map.  A $(G \times \Gm)$-equivariant coherent sheaf on $C_0$ is just a finite-dimensional $(G \times \Gm)$-module.  For such a module $M$, we have $\IC(C_0,M) \cong i_{0*}M[-\frac{1}{2}\dim \cN]$.

More generally, a perverse-coherent sheaf supported set-theoretically (but not necessarily scheme-theoretically) on $C_0$ must be an extension of finitely many such objects $\IC(C_0,M)$.  Such an object is concentrated in cohomological degree $\frac{1}{2}\dim \cN$, and its cohomology sheaf in that degree must have finite-dimensional global sections.  In fact, we have
\[
\left\{
\begin{array}{c}
\text{perverse-coherent sheaves} \\
\text{with $0$-dimensional} \\
\text{support}
\end{array}
\right\}
=
\left\{ {\textstyle\cF[-\frac{1}{2}\dim \cN]} \,\Bigg|\,
\begin{array}{c}
\text{$\cF \in \Cohmix(\cN)$ with} \\
\text{finite-dimensional} \\
\text{global sections}
\end{array}
\right\}.
\]
\end{rmk}

\section{The graded Lusztig--Vogan bijection}\label{sec:gradedlv}

Comparing the two classifications of simple objects in $\PCoh^\Gm(\cN)$, we see that for each $\lambda \in \bX^+$, there is a nilpotent orbit $C$, an irreducible $G^{x_C}$-representation $V$, and an integer $n_{C,V}$ such that
\begin{equation}\label{eqn:graded-lv}
\fC_\lambda \cong \IC(C,V)\la n_{C,V} \ra.
\end{equation}
In particular, if we ignore the grading shift, we see that there is a canonical bijection
\[
\bX^+ \leftrightarrow
\left\{ (C,V) \Big|
\begin{array}{c}
\text{$C$ a nilpotent orbit, and $V$ an} \\
\text{irreducible $G^{x_C}$-representation}
\end{array}
\right\}
\]
known as the \emph{Lusztig--Vogan bijection} (because its existence had been conjectured by Lusztig~\cite{lus:cawg4} and Vogan~\cite{vogan}).  This bijection was established in~\cite{bez:qes} using $G$-equivariant (rather than $(G \times \Gm)$-equivariant) perverse-coherent sheaves, so the problem of determining the integer $n_{C,V}$ in~\eqref{eqn:graded-lv} did not arise.

In~\cite[\S3]{ostrik-Kth}, Ostrik conjectured (at least when $\bk$ has characteristic $0$) that $n_{C,V} = \frac{1}{2}\codim C$.  (See the discussion between Conjectures~2 and~3 in~\cite{ostrik-Kth}.  Most of the conjectures in~\cite{ostrik-Kth} have been proved in~\cite{bezru, bezru-psaf}, but those papers did not determine $n_{C,V}$.)  In this section, we will prove Ostrik's conjecture.  The proof will involve the study of an \emph{opposition} of $G$, i.e., an automorphism $\sigma: G \longrightarrow G$ that satisfies $\sigma(B) = B^+$, preserves the maximal torus $T$, and has the property that $\sigma|_T: T \to T$ is given by $t \mapsto t^{-1}$.  (The existence of such an automorphism follows from~\cite[II.1.16]{jantzen}.  That result describes an antiautomorphism $\tau: G \longrightarrow G$ with similar properties; one can then set $\sigma(t) = \tau(t^{-1})$.)

For any $G$-representation $(V, \varphi: G \to GL(V))$, let $V^\sigma$ denote the representation whose underlying vector space is still $V$, but where the $G$-action is given by $\varphi \circ \sigma: G \to GL(V)$.  The $T$-weights of $V^\sigma$ are the negatives of those of $V$.  As a consequence, we have
\[
V^\sigma \cong V^* \qquad\text{if $V$ is an irreducible $G$-representation.}
\]
(In general, this is is \emph{not} true for reducible representations.)  Under our assumptions (H1)--(H3), the adjoint representation is irreducible and self-dual, so there exists a $G$-equivariant isomorphism
\begin{equation}\label{eqn:fg-twist}
\fg^\sigma \cong \fg.
\end{equation}
Let us fix such an isomorphism once and for all.  By considering the definition of the $G$-action on the rings of coordinate functions on both sides of~\eqref{eqn:fg-twist} (see~\cite[I.2.7]{jantzen}), we obtain an isomorphism $\bk[\fg]^\sigma \cong \bk[\fg]$, and then an isomorphism
\begin{equation}\label{eqn:kn-twist}
\bk[\cN]^\sigma \cong \bk[\cN].
\end{equation}
Using~\eqref{eqn:kn-twist}, we can regard $({-})^\sigma$ as a functor
\[
({-})^\sigma: \Db\Cohmix(\cN) \to \Db\Cohmix(\cN).
\]

\begin{prop}\label{prop:pcoh-twist}
The functor $({-})^\sigma: \Db\Cohmix(\cN) \to \Db\Cohmix(\cN)$ is t-exact for the perverse-coherent t-structure.  It satisfies
\[
\Delta_\lambda^\sigma \cong \Delta_{-w_0\lambda},
\qquad
\nabla_\lambda^\sigma \cong \nabla_{-w_0\lambda},
\qquad
\fC_\lambda^\sigma \cong \fC_{-w_0\lambda}.
\]
\end{prop}
\begin{proof}
We will prove below that
\begin{equation}\label{eqn:aj-twist}
A_\lambda^\sigma \cong A_{-w_0\lambda}.
\end{equation}
Since $\delta_\lambda = \delta_{-w_0\lambda}$ (in the notation of~\eqref{eqn:std-defn}), this immediately implies the formulas for $\Delta_\lambda^\sigma$ and $\nabla_\lambda^\sigma$.  The t-exactness of $\sigma$ follows, as well as the formula for $\fC_\lambda^\sigma$.

Let us prove~\eqref{eqn:aj-twist}. For any $B$-representation $M$, let $M^\sigma$ denote the $B^+$-rep\-resentation given by composing with $\sigma|_{B^+}: B^+ \to B$.  It is clear that
\[
(R\Ind_B^G M)^\sigma \cong R\Ind_{B^+}^G (M^\sigma).
\]
Note that if we regard~\eqref{eqn:kn-twist} as an isomorphism of $B^+$-representations, it restricts to an isomorphism $\fu^\sigma \cong \fu^+$.  Using~\eqref{eqn:aj-sheaf} and~\eqref{eqn:aj-opp}, we find that
\begin{multline*}
(A_\lambda)^\sigma = (R\Ind_B^G (\Sym(\fu^*) \otimes \bk_\lambda))^\sigma
\cong R\Ind_{B^+}^G (\Sym(\fu^*)^\sigma \otimes \bk_\lambda^\sigma) \\
\cong R\Ind_{B^+}^G (\Sym((\fu^+)^*) \otimes \bk_{-\lambda})
\cong A_{-w_0\lambda},
\end{multline*}
as desired.  The fact that this is an isomorphism of complexes of $\bk[\cN]$-modules (and not just of $G$-representations) follows by combining~\eqref{eqn:kn-ind}, \eqref{eqn:kn-opp}, and~\eqref{eqn:kn-twist}.
\end{proof}

The following consequence is well known in characteristic $0$ (cf.~\cite[\S12]{lus:bekt}), but we were unable to find a reference in positive characteristic.

\begin{cor}\label{cor:sigma-nilp}
The automorphism $\sigma: \fg \to \fg$ induced by $\sigma: G \to G$ preserves each nilpotent orbit.
\end{cor}
\begin{proof}
Let $C$ be a nilpotent orbit.  Its image under $\sigma: \fg \to \fg$ is another nilpotent orbit $\sigma(C)$.  For any $\cF \in \Db\Cohmix(\cN)$, the support of $\cF^\sigma$ is obtained by applying $\sigma$ to the support of $\cF$.  Take $\cF$ to be a simple perverse-coherent sheaf $\fC_\lambda$ whose support is $\ol{C}$.  Then the support of $\fC_\lambda^\sigma$ is $\ol{\sigma(C)}$.  But Proposition~\ref{prop:pcoh-twist} and~\eqref{eqn:ic-dual} together tell us that $\fC_\lambda^\sigma \cong \D_\cN(\fC_\lambda)$.  Since $\D_\cN$ preserves supports, we conclude that $\sigma(C) = C$.
\end{proof}

\begin{lem}\label{lem:sigmac}
Let $C$ be a nilpotent orbit, and let $j: C \hookrightarrow \cN$ be the inclusion map.  There is an automorphism $\sigma_C: G^{x_C} \to G^{x_C}$ with the following properties:
\begin{enumerate}
\item It commutes with the action of $\Gm$ on $G^{x_C}$, so that there is an induced automorphism $\id \times \sigma_C: \Gm \ltimes G^{x_C} \to \Gm \ltimes G^{x_C}$.
\item For  $V \in \Rep(\Gm \ltimes G^{x_C})$, regarded as an object of $\Cohmix(C)$ via~\eqref{eqn:indequiv-orbit}, there is a natural isomorphism
\[
\cH^0(j_*(V^{\id \times \sigma_C})) \cong \cH^0((j_*V)^\sigma).
\]
\end{enumerate}
\end{lem}
\begin{proof}
By Corollary~\ref{cor:sigma-nilp}, there exists an $h \in G$ such that $\Ad(h)(x_C) = \sigma(x_C)$.  Let $\sigma' = \Int_{h^{-1}} \circ \sigma: G \to G$.  For any $G$-representation $(M, \varphi: G \to GL(M))$, one can define $M^{\sigma'}$ in the same way as we defined $M^\sigma$ earlier.  However, there is a canonical $G$-equivariant isomorphism
\begin{equation}\label{eqn:sigma-p}
M^\sigma \simto M^{\sigma'}
\qquad\text{given by}\qquad
v \mapsto \varphi(h^{-1})v.
\end{equation}
The same formula gives an isomorphism $\cF^\sigma \cong \cF^{\sigma'}$ for any $\cF \in \Db\Cohmix(\cN)$.

Now, consider the cocharacter $\phi': \Gm \to G$ given by $\phi' = \Int_{h^{-1}} \circ \sigma \circ \phi_{x_C}$.  A straightforward calculation shows that $\phi'$ is an associated cocharacter for $x_C$.  By~\cite[Lemma~5.3]{jantzen-nilp}, there is some element $h' \in G^{x_C}$ such that $\Int_{h'} \circ \phi' = \phi_{x_C}$.  By replacing $h$ by $h(h')^{-1}$, we may assume without loss of generality that $\phi' = \phi_{x_C}$.  In other words,
\begin{equation}\label{eqn:sigma-phi}
\sigma' \circ \phi_{x_C} = \phi_{x_C}.
\end{equation}
Since $\sigma'$ fixes $x_C$, it preserves $G^{x_C}$.  Let $\sigma_C = \sigma'|_{G^{x_C}}: G^{x_C} \to G^{x_C}$.  It follows from~\eqref{eqn:sigma-phi} that $\sigma_C$ commutes with the action of $\Gm$ on $G^{x_C}$, so there is a well-defined automorphism
\[
\id \times \sigma_C: \Gm \ltimes G^{x_C} \to \Gm \ltimes G^{x_C}.
\] 
Moreover, our set-up implies that the following diagram commutes:
\begin{equation}\label{eqn:sigma-commute}
\begin{tikzcd}[column sep=7.2em]
\Gm \ltimes G^{x_C} \ar[d, "\id \times \sigma_C"'] \ar[r, "{(z,g) \mapsto (\phi_{x_C}(z)g, z)}"]
 & G \times \Gm \ar[d, "\sigma' \times \id"] \\ 
\Gm \ltimes G^{x_C}  \ar[r, "{(z,g) \mapsto (\phi_{x_C}(z)g, z)}"] 
 & G \times \Gm 
\end{tikzcd}
\end{equation}

For the remainder of this proof, we use the notation $j_*$ to mean the \emph{nonderived} push-forward functor.  Thus, the statement we wish to prove is simply $j_*(V^{\sigma_C}) \cong (j_*V)^\sigma$.

Let us recall how to compute global sections under the equivalence~\eqref{eqn:indequiv-orbit}.  If $H$ is an algebraic group and $K \subset H$ is a closed subgroup, and if $M$ is a $K$-module, then the global sections of the corresponding vector bundle on the homogeneous space $H/K$ are given by $\Ind_K^H M$.  To apply this in our situation, we combine the equivalence~\eqref{eqn:indequiv-orbit} with the isomorphism~\eqref{eqn:graded-centralizer}, and we conclude that
\[
\Gamma(j_*V) \cong \Ind_{\Gm \ltimes G^{x_C}}^{G \times \Gm} V.
\]
It follows immediately from the commutativity of~\eqref{eqn:sigma-commute} (along with~\eqref{eqn:sigma-p}) that
\begin{equation}\label{eqn:sigma-gamma}
\Gamma(j_*(V^{\id \times \sigma_C})) \cong \Gamma((j_*V)^{\sigma'}) \cong \Gamma((j_*V)^\sigma).
\end{equation}
These are isomorphisms of $(G \times \Gm)$-representations.  To show that $j_*(V^{\id \times \sigma_C})$ and $(j_*V)^\sigma$ are isomorphic as sheaves on the affine variety $\ol{C}$, we must show that~\eqref{eqn:sigma-gamma} is an isomorphism of $\bk[\ol{C}]$-modules.

Note that $\Gamma(j_*\bk)$ naturally has the structure of a ring: it is the ring of regular functions on $C$.  As a special case of~\eqref{eqn:sigma-gamma}, we have ring isomorphisms
\begin{equation}\label{eqn:sigma-ring}
\Gamma(j_*\bk) = \Gamma(j_*(\bk^{\id \times \sigma_C})) \cong \Gamma((j_*\bk)^\sigma).
\end{equation}
Moreover, $\Gamma(j_*(V^{\id \times \sigma_C}))$ is naturally a $\Gamma(j_*(\bk^{\id \times \sigma_C}))$-module, and $\Gamma((j_*V)^\sigma)$ is naturally a $\Gamma((j_*\bk)^\sigma)$-module.  It is readily seen by unwinding the definitions that these module structures are compatible with the ring isomorphisms in~\eqref{eqn:sigma-ring}.  In other words,~\eqref{eqn:sigma-gamma} is an isomorphism of $\Gamma(j_*\bk)$-modules.

Since $\ol{C} \smallsetminus C$ has codimension at least $2$ in $\ol{C}$, the natural map $\bk[\ol{C}] \to \Gamma(j_*\bk)$ is injective.  (In fact, $\Gamma(j_*\bk)$ is the normalization of $\bk[\ol{C}]$.)  We conclude that~\eqref{eqn:sigma-gamma} is an isomorphism of $\bk[\ol{C}]$-modules, as desired.
\end{proof}

\begin{lem}\label{lem:ic-sigmac}
Let $\sigma_C$ be as in Lemma~\ref{lem:sigmac}.  For $V \in \Rep(\Gm \ltimes G^{x_C})$, there is a natural isomorphism $\IC(C,V^{\id \times \sigma_C}) \cong \IC(C,V)^\sigma$.
\end{lem}
\begin{proof}
This proof requires unpacking the construction of the $\IC$ functor.  Let $j: C \hookrightarrow \ol{C}$ be the inclusion map.  First,~\cite[Lemma~4.3]{ab:pcs} introduces a functor $J_{!*}: \Db\Cohmix(\ol{C}) \to \PCoh^\Gm(\ol{C})$ with various desirable properties (some of which will be recalled below).  Second, the discussion following~\cite[Lemma~4.3]{ab:pcs} shows that there is a unique functor $j_{!*}: \PCoh^\Gm(C) \to \PCoh^\Gm(\ol{C})$ characterized by the property that there is a natural isomorphism
\[
J_{!*}(\cF) \cong j_{!*}(\cF|_C)
\]
for $\cF \in \PCoh^\Gm(\ol{C})$.  Finally, the $\IC$ functor is defined by
\[
\textstyle
\IC(C,V) = j_{!*}(V[-\frac{1}{2}\codim C]).
\]

As in the proof of Lemma~\ref{lem:sigmac}, let us use $j_*$ to mean the nonderived pushforward.  From the remarks above, we have
\[
\textstyle
\IC(C,V) \cong J_{!*}((j_*V)[-\frac{1}{2}\codim C]).
\]
Here, we are implicitly using the fact that $j_*V$ is a coherent (rather than merely quasicoherent) sheaf.  This can be deduced from~\cite[Corollary~3.12]{ab:pcs}, using the fact that $\ol{C} \smallsetminus C$ has codimension at least $2$ in $\ol{C}$.

Next, the functor $J_{!*}$ is defined in~\cite{ab:pcs} as the composition of two truncation functors: $\tau^-_{\le 0} \circ \tau^+_{\ge 0}$.  (We refer to~\cite{ab:pcs} for the definitions of these functors.)  A routine truncation functor argument shows that the order of composition can be reversed: we may instead write $J_{!*} \cong \tau^+_{\ge 0} \circ \tau^-_{\le 0}$.  Finally, because $j_*V$ is a coherent sheaf (rather than a complex of sheaves), it follows from the definitions that
\[
\textstyle
(j_*V)[-\frac{1}{2}\codim C] \in {}^{p^-}\Db\Cohmix(\ol{C})^{\le 0}.
\]
We can therefore omit the $\tau^-_{\le 0}$: we have
\[
\textstyle
J_{!*}((j_*V)[-\frac{1}{2}\codim C]) \cong \tau^+_{\ge 0}((j_*V)[-\frac{1}{2}\codim C]).
\]
This object fits into a truncation distinguished triangle:
\[
\textstyle
\tau^+_{\le -1}((j_*V)[-\frac{1}{2}\codim C]) \to (j_*V)[-\frac{1}{2}\codim C] \to \IC(C,V) \to.
\]
The second and third terms become isomorphic after restriction to $C$, so the first term must be supported on $\ol{C} \smallsetminus C$.  Now apply $\sigma$ to this distinguished triangle, and use Lemma~\ref{lem:sigmac} to obtain
\begin{equation}\label{eqn:ic-trunc}
\textstyle
(\tau^+_{\le -1}((j_*V)[-\frac{1}{2}\codim C]))^\sigma \to (j_*(V^{\id \times \sigma_C}))[-\frac{1}{2}\codim C] \to \IC(C,V)^\sigma \to.
\end{equation}
By Proposition~\ref{prop:pcoh-twist}, the last term is still simple perverse-coherent sheaf.  It is still supported on $\ol{C}$, so we must have
\[
\IC(C,V)^\sigma \cong \IC(C,V')
\]
for some $V' \in \Rep(\Gm \ltimes G^{x_C})$.  The first term of~\eqref{eqn:ic-trunc} is still supported on $\ol{C} \smallsetminus C$, so to determine $V'$, we must examine the restriction to $C$ of the middle term of~\eqref{eqn:ic-trunc}.  This shows us that $V' \cong V^{\id \times \sigma_C}$, as desired.
\end{proof}

\begin{thm}\label{thm:graded-lv}
Let $\lambda \in \bX^+$, let $C$ be a nilpotent orbit, and let $V$ be an irreducible $G^{x_C}$-representation. Suppose $\lambda$ corresponds to $(C,V)$ under the Lusztig--Vogan bijection. Then $\fC_\lambda \cong \IC(C,V)\la \frac{1}{2}\codim C\ra$.
\end{thm}
\begin{proof}
For brevity, let $n = n_{C,V}$ be the integer such that $\fC_\lambda \cong \IC(C,V)\la n\ra$, or $\fC_\lambda\la -n\ra \cong \IC(C,V)$.  Using Proposition~\ref{prop:pcoh-twist} with~\eqref{eqn:ic-dual} and Lemma~\ref{lem:ic-dual-calc}, we see that
\begin{multline*}
\IC(C,V)^\sigma \cong \fC_\lambda^\sigma\la -n \ra \cong \D(\fC_\lambda)\la -n\ra \\
\cong \D(\IC(C,V)\la n\ra)\la -n\ra
\cong \IC(C,V^*)\la \codim C -2n\ra.
\end{multline*}
Comparing this with Lemma~\ref{lem:ic-sigmac}, we find that $V^{\id \times \sigma_C} \cong V^*\la \codim C - 2n\ra$.  Since the $\Gm$-action on $V^{\id \times \sigma_C}$ is trivial, we deduce that $\codim C - 2n = 0$, and hence that $n = \frac{1}{2}\codim C$.
\end{proof}

\begin{rmk}\label{rmk:ic-sigma}
In the proof of Theorem~\ref{thm:graded-lv}, we saw that if $V$ is an irreducible $G^{x_C}$-representation, then $\IC(C,V)^\sigma \cong \IC(C,V^*)$.  However, in general, this does not hold if $V$ is reducible.
\end{rmk}

\section{Conventions for \texorpdfstring{$PGL_3$}{PGL3}}\label{sec:notation}

From now on, we assume that $G = PGL_3$, and that the characteristic of $\bk$ is not $2$ or $3$ (so that (H1)--(H3) hold). Let $T \subset G$ be the maximal torus consisting of diagonal matrices. We identify the weight lattice $\bX$ with $\{ (a,b,c) \in \Z^3 \mid a+ b+c = 0 \}$ in the usual way: explicitly, the character $\lambda = (a,b,c)$ is given by
\[
\lambda \left(
\begin{psmallmatrix}
t_1 && \\ & t_2 & \\ && t_3
\end{psmallmatrix} \right)
 = t_1^a t_2^b t_3^c.
\]
Let $B \subset G$ be the Borel subgroup consisting of lower-triangular matrices. Then the set of dominant weights is
\[
\bX^+ = \{ (a,b,c) \mid \text{$a + b + c = 0$ and $a \ge b \ge c$} \}.
\]
Let $W$ be the Weyl group of $G$.  For $\lambda \in \bX^+$, let $\delta_\lambda = \min \{ \ell(w) \mid w\lambda \in -\bX^+ \}$.  

We label the three $G$-orbits in $\cN$ by partitions of $3$: $C_{[3]}$, $C_{[2,1]}$, and $C_{[1,1,1]}$.  They satisfy the following closure relations:
\[
\{0\} = C_{[1,1,1]} \subset \ol{C_{[2,1]}} \subset \ol{C_{[3]}}.
\]
These orbits have dimensions $0$, $4$, and $6$, respectively.  For each partition $\rd \vdash 3$, we choose a representative $x_\rd \in C_\rd$ and an associated cocharacter $\phi_\rd: \Gm \to G$ as shown below.  This table also shows the groups $G^\rd_\red$, defined as in~\eqref{eqn:gxred-defn}.
\begin{align*}
x_{[1,1,1]} &= 0 &
  \phi_{[1,1,1]}(z) &= 1 &
  G^{[1,1,1]}_\red &= G \\
x_{[2,1]} &= \begin{psmallmatrix} 0 & 0 & 0 \\ 0 & 0 & 0 \\ 1 & 0 & 0 \end{psmallmatrix} &
  \phi_{[2,1]}(z) &= \begin{psmallmatrix} z^{-1} & &  \\  & 1 &  \\  & & z \end{psmallmatrix} &
  G^{[2,1]}_\red &= \left\{ \begin{psmallmatrix} 1 & &  \\  & t &  \\  & & 1 \end{psmallmatrix} \right\} \cong \Gm \\
x_{[3]} &= \begin{psmallmatrix} 0 & 0 & 0 \\ 1 & 0 & 0 \\ 0 & 1 & 0 \end{psmallmatrix} &
  \phi_{[3]}(z) &= \begin{psmallmatrix} z^{-2} & &  \\  & 1 &  \\  & & z^2 \end{psmallmatrix} &
  G^{[3]}_\red &= \{1\}
\end{align*}
We can likewise consider the group $(G \times \Gm)^\rd_\red$, defined to be the stabilizer in $G \times \Gm$ of $\phi_\rd$, respectively.  This group is a Levi factor of $(G \times \Gm)^\rd$.  Following~\eqref{eqn:graded-centralizer}, there is an isomorphism
\[
\Gm \times G_{\red}^\rd \simto (G \times \Gm)^\rd_{\red}
\qquad\text{given by}\qquad
(z,g) \mapsto (\phi_{\rd}(z)g, z).
\]

The Lusztig--Vogan bijection for $GL_n$ has been determined in~\cite{achar-phd, achar-Kth}.  (Those sources assume that $\bk = \C$, but by~\cite{ahr2}, these results hold in positive characteristic as well.)  By restricting to appropriate subsets on both sides, one obtains the Lusztig--Vogan bijection for $PGL_n$.  The resulting bijection for $PGL_3$ is recorded in Table~\ref{tab:lv}.

\begin{table}
\begin{center}
 \begin{tabular}{||c |c |c | c||} 
 \hline
 $\rd \vdash 3$  & $G^{\rd}_{\red}$ & $\irr \in \Irr(G^{\rd})$ & $\lambda  \in \bX^+$  \\ [0.5ex] 
 \hline\hline
 $[3]$ & $\{\id\}$ & $\bk_{\mathrm{triv}}$ & $\lambda = (0,0,0)$ \\ 
 \hline
 $[2,1]$ & $\Gm$ & $\bk_{a}$, $a \in \Z$ & $\lambda = 
 						\begin{cases} 
						(x+1, x+1, -2x-2) & \text{if $a = 2x+1\geq 0$,} \\
						(x+1,x, -2x-1) & \text{if $a = 2x \geq 0$},\\
						(-2x-2, x, x) &\text{if $a = 2x+1 \leq 0$,}\\
						(-2x-1, x, x-1) &\text{if $a=2x \leq 0$}
 						\end{cases}$  \\
 \hline
  $[1,1,1]$ & $G$ & $\irr(a,b,c)$ & $\lambda = (a-2,b,c+2)$ \\ [1ex] 
 \hline
\end{tabular}
\end{center}
\caption{The Lusztig--Vogan bijection for $PGL_3$}\label{tab:lv}
\end{table}

\section{A resolution of the middle orbit}\label{sec:geometry}

In this section, we focus on the middle orbit $C_{[2,1]}$.  Let $\alpha_0 = (1,0,-1)$ be the highest root, and let
\[
\fu_{-\alpha_0} =  \bk x_{[2,1]} \subset \fg
\]
be the corresponding root space.  The group $B$ acts on $\fu_{-\alpha_0}$ by the adjoint action.  Form the vector bundle
\[
\cV = G \times^B \fu_{-\alpha_0},
\]
and let $\pi_\cV: \cV \to \ol{C_{[2,1]}}$ be the map $\pi_\cV(g,x) = \Ad(g)(x)$.  There is an obvious action of $G \times \Gm$ on $\cV$ (where $\Gm$ acts on $\fu_{-\alpha_0}$ as usual by $z \cdot x = z^{-2}x$), and the map $\pi_\cV$ is $(G \times \Gm)$-equivariant.

For $\lambda \in \bX$, let $\cO_{G/B}(\lambda)$ be the corresponding line bundle on $G/B$.  We regard this as a $(G \times \Gm)$-equivariant coherent sheaf by having $\Gm$ act trivially.  Let $p: \cV \to G/B$ be the projection map given by $p(g,x) = gB$.  We then set
\[
\cO_\cV(\lambda) = p^*\cO_{G/B}(\lambda) \in \Db\Cohmix(\cV).
\]

\begin{lem}\label{lem:rational-resolution}
The map $\pi_{\cV}: G \times^B \fu_{{-\alpha_0}} \longrightarrow  \ol{C}$ is a resolution of singularities of $\ol{C}$. 
\end{lem}
\begin{proof}
We observe immediately that $G \times^B \fu_{{-\alpha_0}}$ is smooth and irreducible, and that
$\dim G \times^B \fu_{{-\alpha_0}} = 4 = \dim \ol{C}$.   To 
show that it is birational, we now only have to verify that 
\[
\pi_{\cV}|_{\tilde{C}}: \tilde{C} \longrightarrow C
\]
is an isomorphism, where $\tilde{C} = \pi_{\cV}^{-1}(C)$. This follows from \cite[Lemma 8.8 and Remark 8.8]{jantzen-nilp}, where 
we set $V = \fg$, $M= \fu_{-\alpha_0}$ and $P=B$, so that $\mathcal{P}_M \cong G \times^B \fu_{{-\alpha_0}}$. 
\end{proof}

\begin{lem}\label{lem:dualzing}
The canonical bundle of $\cV$ is given by $\omega_\cV \cong \cO_\cV({-\alpha_0})\la -2\ra$.
\end{lem}
\begin{proof}
We generalize the notation $\cO_{G/B}(\lambda)$ as follows: for any $(B \times \Gm)$-module $M$, let $\cO_{G/B}(M)$ denote the corresponding $(G \times \Gm)$-equivariant locally free sheaf on $G/B$, and let $\cO_\cV(M) = p^*\cO_{G/B}(M)$.

Using~\cite[Proposition II.8.11]{hartshorne} and the fact that $\cV$ and $G/B$ are smooth, we have a short exact sequence
\[
0 \longrightarrow p^*\Omega_{G/B} \longrightarrow \Omega_{\cV} \longrightarrow \Omega_{\cV/ (G/B)} \longrightarrow 0.
\]
These three sheaves are locally free of ranks $3$, $4$, and $1$, respectively. By applying the reasoning from \cite[II.4.1]{jantzen} to $\cV$, we have $ p^*\Omega_{G/B} = \cO_{\cV}((\fg/\fb)^*)$ and $\Omega_{\cV/ (G/B)} = \cO_{\cV}(\fu_{{-\alpha_0}}^*)$.  From these observations, we have
\begin{multline*}
\textstyle
\omega_\cV = \bigwedge^4 \Omega_\cV \cong \bigwedge^3(p^*\Omega_{G/B}) \otimes \bigwedge^1 \Omega_{\cV/(G/B)} \\
\textstyle
\cong \bigwedge^3 \cO_\cV((\fg/\fb)^*) \otimes \cO_\cV(\fu_{-\alpha_0}^*)
\cong \cO_\cV(\bigwedge^3(\fg/\fb)^* \otimes \fu_{-\alpha_0}^*).
\end{multline*}
Now, $B$ acts on the $1$-dimensional representation $\bigwedge^3(\fg/\fb)^* \otimes \fu_{-\alpha_0}^*$ with weight ${-\alpha_0}$.  Next, $\Gm$ acts trivially on $G/B$, so for the purposes of this calculation, it also acts trivially on the cotangent space $(\fg/\fb)^*$.  It acts with weight $2$ on $\fu_{-\alpha_0}^*$, so $\omega_\cV \cong \cO_\cV({-\alpha_0})\la -2\ra$, as desired.
\end{proof}

\begin{lem}\label{lem:oc-shriek}
We have $\pi_\cV^!\cO_{\ol{C_{[2,1]}}} \cong \cO_\cV({-\alpha_0})\la 2\ra$.
\end{lem}
\begin{proof}
Let $a: \ol{C_{[2,1]}} \to \Spec \bk$ and $a_\cV: \cV \to \Spec \bk$ be the structure maps.  Since $\ol{C_{[2,1]}}$ is Gorenstein~(by~\cite[Theorem~5.3.2]{bk}) and $\dim C_{[2,1]} = 4$, we can use Proposition~\ref{prop:gorenstein-dc} to calculate as follows:
\[
\pi_\cV^!\cO_{\ol{C_{[2,1]}}} \cong \pi_\cV^! a^!\cO_{\Spec \bk}[-4]\la 4\ra \cong a_\cV^!\cO_{\Spec \bk}[-4]\la 4\ra
\cong \omega_\cV\la 4\ra \cong \cO_\cV({-\alpha_0})\la 2\ra,
\]
as desired.
\end{proof}

Define the Serre--Grothendieck duality functor on $\cV$ by
\[
\D_\cV = R\shHom({-}, \cO_\cV({-\alpha_0})[-2]\la 4\ra).
\]
This choice satisfies
\[
\pi_{\cV*} \circ \D_\cV \cong \D_{\ol{C_{[2,1]}}} \circ \pi_{\cV*}.
\]
To see this, note that by Corollary~\ref{cor:orbit-dc}, the right-hand side is given by
\[
R\shHom(\pi_{\cV*}({-}), \cO_{\ol{C_{[2,1]}}}[-2]\la 2\ra) \cong \pi_{\cV*} R\shHom({-},\pi_\cV^!\cO_{\ol{C_{[2,1]}}}[-2]\la 2\ra),
\]
and then use Lemma~\ref{lem:oc-shriek}.

As an immediate consequence, we have the following

\begin{lem}\label{lem:V-properties}
For any $\lambda \in \bX$, we have
\begin{align*}
\D_\cV(\cO_\cV(\lambda)\la n\ra[k]) &\cong \cO_\cV(-\alpha_0-\lambda)\la 4-n\ra[-2-k], \\
\D_\cN(\pi_{\cV*}\cO_\cV(\lambda)\la n\ra[k]) &\cong \pi_{\cV*}\cO_\cV(-\alpha_0-\lambda)\la 4-n\ra[-2-k].
\end{align*}
\end{lem}

\section{Determination of the simple perverse-coherent sheaves}\label{sec:mainthm}

\newcommand{\ts}{\;\,}

In this section, we will work with representations of $G^{[2,1]}$ and $\Gm \ltimes G^{[2,1]}$.  Recall that $G^{[2,1]}_\red$ is isomorphic to $\Gm$.  Let $\bk_n$ denote the irreducible $G^{[2,1]}_\red$-representation of weight $n$.  Thus, a general irreducible $(\Gm \ltimes G^{[2,1]})$-representation can be written in the form
\[
\bk_n\la m\ra,
\]
where $n$ is the weight of the $G^{[2,1]}_\red$-action, and $-m$ is (as usual) the weight of the other copy of $\Gm$.

\begin{lem}\label{lem:wt-restrict}
Let $\lambda = (a,b,c) \in \bX$, and consider the line bundle $\cO_\cV(\lambda)$.  The object $(\pi_{\cV*}\cO_\cV(\lambda))|_{C_{[2,1]}}$ is the line bundle corresponding to the $(\Gm \ltimes G^{[2,1]})$-representation $\bk_b\la a-c\ra$.
\end{lem}
\begin{proof}
By Lemma~\ref{lem:rational-resolution}, we can identify $C_{[2,1]}$ with its preimage under $\pi_\cV$.  To describe $\cO_\cV(a,b,c)|_{C_{[2,1]}}$, we must compute the restriction of the $B$-representation of weight $(a,b,c)$ to $\Gm \times G^{[2,1]}_\red$ via~\eqref{eqn:graded-centralizer}.  It is clear that $G^{[2,1]}_\red$ acts with weight $b$.  The weight of the $\Gm$-action is given by pairing the cocharacter $\phi_{[2,1]} = (-1,0,1)$ with $(a,b,c)$ to obtain $c-a$.
\end{proof}

The following lemma describes the graded $G$-module structure of certain coherent sheaves on $\cN$.  In the tables, the column headings indicate the grading degree (i.e., the weight of the $\Gm$-action), and the entries are $G$-representations.  For the latter, we use the following notation from~\cite{jantzen}: for $\lambda \in \bX$ and $i \in \Z$, we write
\[
H^i(\lambda) = R^i \Ind_B^G \bk_\lambda.
\]

\begin{lem}\label{lem:piv-cohom}
Let $a \ge 0$, and let $\lambda_a = (a,a,-2a) \in \bX$.
\begin{enumerate}
\item We have $\cH^i(\pi_{\cV*}\cO_\cV(\lambda_a)) = 0$ unless $i = 0$.  As a graded $G$-module, the sheaf $\cH^0(\pi_{\cV*}\cO_\cV(\lambda_a))$ is given by
\[
\scalebox{0.9}{\hbox{$
\begin{array}{@{}c@{\ts}c@{\ts}c@{\ts}c@{\ts}c@{\ts}c@{\ts}c@{\ts}c@{\ts}c@{\ts}c@{\ts}c@{}}
& 0 & 2 & 4 & \cdots & 2r & \cdots  \\
\cH^0: &H^0(\lambda_a) & H^0(\lambda_a+\alpha_0) & H^0(\lambda_a+2\alpha_0) & \cdots & H^0(\lambda_a+r\alpha_0) & \cdots
\end{array}$}}
\]
\item We have $\cH^i(\pi_{\cV*}(\cO_\cV(-\alpha_0-\lambda_a)) = 0$ unless $0 \le i \le 2$.  As graded $G$-modules, the sheaves $\cH^i(\pi_{\cV*}(\cO_\cV(-\alpha_0-\lambda_a))$ are given by\label{it:piv-minus}
\[
\scalebox{0.73}{\hbox{$
\begin{array}{@{}c@{\ts}c@{\ts}c@{\ts}c@{\ts}c@{\ts}c@{\ts}c@{\ts}c@{\ts}c@{\ts}c@{\ts}c@{\ts}c@{}}
& 0 & 2 & 4 & \cdots & 6a-2 & 6a & 6a+2 & 6a+4 & \cdots & 6a+2r & \cdots  \\
\cH^2: &0 &H^2(\lambda_a^\sigma - 3a\alpha_0) & H^2(\lambda_a^\sigma - (3a-1)\alpha_0) & \cdots & H^2(\lambda_a^\sigma-2\alpha_0) & 0 \\
\cH^1: &0 &H^1(\lambda_a^\sigma - 3a\alpha_0) & H^1(\lambda_a^\sigma - (3a-1)\alpha_0) & \cdots & H^1(\lambda_a^\sigma-2\alpha_0) & 0 \\
\cH^0: &0 &&&&& 0 &H^0(\lambda_a^\sigma) & H^0(\lambda_a^\sigma+\alpha_0) & \cdots & H^0(\lambda_a^\sigma+r\alpha_0) & \cdots
\end{array}$}}
\]
where $\lambda_a^\sigma = -w_0\lambda_a = (2a,-a,-a)$.  In particular, $\cH^1(\pi_{\cV*}(\cO_\cV(-\alpha_0-\lambda_a))$ and $\cH^2(\pi_{\cV*}(\cO_\cV(-\alpha_0-\lambda_a)))$ have $0$-dimensional support.
\end{enumerate}
\end{lem}
The notation $\lambda_a^\sigma$ above is related to the use of $({-})^\sigma$ in Section~\ref{sec:gradedlv} by the fact that $H^i(\lambda_a)^\sigma \cong H^i(\lambda_a^\sigma)$ (cf.~the proof of Proposition~\ref{prop:pcoh-twist}).

\begin{proof}
For any weight $\mu$, the $G$-module structure on global sections of $\pi_{\cV*}\cO_\cV(\mu)$ is given by 
\[
R\Gamma(\pi_{\cV*}\cO_{\cV}(\mu)) \cong R\Ind_B^G \left(\bigoplus_{n \ge 0} \Sym^n(\fu_{-\alpha_0}^*) \otimes \bk_\mu \right) \cong \bigoplus_{n \ge 0} R\Ind_B^G \bk_{\mu + n\alpha_0}.
\]
Since $\Gm$ acts on $\fu_{-\alpha_0}$ with weight $-2$, it acts on $\Sym^n(\fu_{-\alpha_0}^*)$ with weight $2n$.  In other words, the $2n$-th graded piece of $\cH^i(\pi_{\cV*}\cO_\cV(\mu))$ is $H^i(\mu+n\alpha_0)$.  Since $\dim G/B = 3$, we clearly have $H^i(\mu+n\alpha_0) = 0$ unless $0 \le i \le 3$.  

Suppose now that $\mu = \lambda_a$.  The weights $\lambda_a + n\alpha_0$ are dominant for all $n \ge 0$, so $\cH^i(\pi_{\cV*}\cO_\cV(\mu)) = 0$ for $i > 0$.

Next, let us take $\mu = -\alpha_0 - \lambda_a$.  Note that
\begin{multline*}
-\alpha_0 -\lambda_a + n\alpha_0 = (-a-1+n,-a,2a+1-n) \\
= (2a+n-3a-1,-a,-a-n+3a+1) = \lambda_a^\sigma + (n-3a-1)\alpha_0.
\end{multline*}
In other words, the $2n$-th graded piece of $\cH^i(\pi_{\cV*}\cO_\cV(-\alpha_0-\lambda_a))$ is $H^i(\lambda_a^\sigma + (n-3a-1)\alpha_0)$.  For $n \ge 0$, the weights $\lambda_a^\sigma + (n-3a-1)\alpha_0$ are never of the form $w_0\mu - 2\rho$ with $\mu \in \bX^+$, where $2\rho = (2,0,-2)$.  By Serre duality (see, for instance,~\cite[Eq.~II.4.2.(9)]{jantzen}, along with~\cite[Proposition~II.2.6]{jantzen}), we deduce that $H^3(\lambda_a^\sigma + (n-3a-1)\alpha_0) = 0$ for all $n$, and hence $\cH^3(\pi_{\cV*}\cO_\cV(\mu)) = 0$.

Next, if $n = 0$ or $n = 3a$, the weight $\lambda_a^\sigma + (n-3a-1)\alpha_0$ pairs with one of the simple coroots to give $-1$.  In these two cases, by~\cite[Proposition~II.5.4]{jantzen}, we have $R\Ind_B^G \bk_{\lambda_a^\sigma + (n-3a-1)\alpha_0} = 0$.  If $0 < n < 3a$, then $\lambda_a^\sigma + (n-3a-1)\alpha_0$ is not dominant, so $H^0(\lambda_a^\sigma + (n-3a-1)\alpha_0) = 0$.  Finally, if $n > 3a$, then $\lambda_a^\sigma + (n-3a-1)\alpha_0$ is dominant, so $H^i(\lambda_a^\sigma+(n-3a-1)\alpha_0) = 0$ for $i > 0$.

These calculations show that $\cH^1(\pi_{\cV*}(\cO_\cV(-\alpha_0-\lambda_a))$ and $\cH^2(\pi_{\cV*}(\cO_\cV(-\alpha_0-\lambda_a))$ are finite-dimensional $G$-modules, so as coherent sheaves, they have $0$-dimensional support.
\end{proof}

\begin{rmk}\label{rmk:piv-0}
When $\bk$ has characteristic $0$, we can refine Lemma~\ref{lem:piv-cohom}\eqref{it:piv-minus} somewhat.  By the Borel--Weil--Bott theorem, $H^i(\lambda_a^\sigma + (n-3a-1)\alpha_0)$ is nonzero for at most one $i$.  In fact, by examining the signs of the pairing of $\lambda_a^\sigma + (n-3a-1)\alpha_0$ with various coroots, one can see that $H^i(\lambda_a^\sigma + (n-3a-1)\alpha_0)$ is nonzero only for $i = 2$ if $0 < n < \frac{3a}{2}$, and only for $i = 1$ if $\frac{3a}{2} < n < 3a$.  Moreover, if $n = \frac{3a}{2}$, then $H^i(\lambda_a^\sigma + (n-3a-1)\alpha_0) = 0$ for all $i$.

We record this observations as follows.  If $\bk$ has characteristic $0$ and $a \ge 0$ is even, then $\pi_{\cV*}(\cO_\cV(-\alpha_0-\lambda_a)$ is given by
\[
\scalebox{0.73}{\hbox{$
\begin{array}{@{}c@{\ts}c@{\ts}c@{\ts}c@{\ts}c@{\ts}c@{\ts}c@{\ts}c@{\ts}c@{\ts}c@{\ts}c@{\ts}c@{}}
 & 2 & \cdots & 3a-2 & 3a & 3a+2 & \cdots & 6a-2 & 6a & 6a+2 & 6a+4 & \cdots  \\
\cH^2: &H^2(\lambda_a^\sigma - 3a\alpha_0) & \cdots & H^2(\lambda_a^\sigma - \frac{3a+4}{2}\alpha_0) &0 &  &  &  & 0 \\
\cH^1: & & &  & 0 & H^1(\lambda_a^\sigma - \frac{3a}{2}\alpha_0) & \cdots & H^1(\lambda_a^\sigma-2\alpha_0) & 0 \\
\cH^0: &&&& 0 &&&& 0 &H^0(\lambda_a^\sigma) & H^0(\lambda_a^\sigma+\alpha_0) & \cdots 
\end{array}$}}
\]
If $\bk$ has characteristic $0$ and $a \ge 1$ is odd, then $\pi_{\cV*}(\cO_\cV(-\alpha_0-\lambda_a))$ is given by
\[
\scalebox{0.73}{\hbox{$
\begin{array}{@{}c@{\ts}c@{\ts}c@{\ts}c@{\ts}c@{\ts}c@{\ts}c@{\ts}c@{\ts}c@{\ts}c@{\ts}c@{}}
& 2 & \cdots & 3a-1 & 3a+1 & \cdots & 6a-2 & 6a & 6a+2 & 6a+4 & \cdots  \\
\cH^2: &H^2(\lambda_a^\sigma - 3a\alpha_0) & \cdots & H^2(\lambda_a^\sigma - \frac{3a+3}{2}\alpha_0) &  &  &  & 0 \\
\cH^1:  & & &  & H^1(\lambda_a^\sigma - \frac{3a+1}{2}\alpha_0) & \cdots & H^1(\lambda_a^\sigma-2\alpha_0) & 0 \\
\cH^0: &&&&&&& 0 &H^0(\lambda_a^\sigma) & H^0(\lambda_a^\sigma+\alpha_0) & \cdots 
\end{array}$}}
\]
\end{rmk}

\begin{prop}\label{prop:ic-a-calc}
For any $a \ge 0$, there is an isomorphism
\[
\IC([2,1],\bk_{-a}) \cong \tau^{\le 2}(\pi_{\cV*}\cO_\cV(-\alpha_0-\lambda_a)[-1])\la 3a+2\ra,
\]
where $\tau^{\le 2}$ denotes truncation with respect to the standard $t$-structure.
\end{prop}

\begin{proof}
We begin by showing that $\pi_{\cV*}\cO_\cV(\lambda_a)[-1]$ and $\D \pi_{\cV*}(\cO_{\cV}(\lambda_a)[-1])$ are per\-verse-coherent sheaves.  To do this, we must check the dimension-support conditions given in 
\cite[Theorem~4.6(2)]{achar}.  
By Lemma~\ref{lem:piv-cohom}, the cohomology of $\pi_{\cV*}\cO_{\cV}(\lambda_a)[-1]$ is concentrated in degree 1, and is supported on $\ol{C_{[2,1]}}$, which has dimension~$4$. By Lemmas~\ref{lem:V-properties} and~\ref{lem:piv-cohom}, the cohomology of
\[
\D(\pi_{\cV*}\cO_{\cV}(\lambda_a)[-1]) \cong \pi_{\cV*}\cO_\cV(-\alpha_0-\lambda_a)\la 4\ra[-1]
\]
satisfies $\cH^i(\D(\pi_{\cV*}\cO_{\cV}(\lambda_a)[-1])) =0$ unless $1 \leq i \leq 3$.  For $i = 1$, the support of this sheaf again has dimension $4$.  For $i = 2,3$, Lemma~\ref{lem:piv-cohom} says that these sheaves have $0$-dimensional support.
Therefore,  $\pi_{\cV*}\cO_{\cV}(\lambda_a)[-1]$ and $\D(\pi_{\cV*}\cO_{\cV}(\lambda_a)[-1])$ satisfy the dimension-support conditions, and are thus perverse coherent.

Next, consider the following truncation distinguished triangle (taken with respect to the standard $t$-structure):
\begin{multline}\label{eq:ses-1}
\tau^{\le 2}(\pi_{\cV*}\cO_\cV(-\alpha_0-\lambda_a)[-1]) \to
\pi_{\cV*}\cO_\cV(-\alpha_0-\lambda_a)[-1] \to \\
\tau^{\ge 3}(\pi_{\cV*}\cO_\cV(-\alpha_0-\lambda_a)[-1]) \to.
\end{multline}
We have just seen that the middle term is perverse-coherent.  The cohomology sheaves of the first and third terms still obey a one-sided version of the conditions in~\cite[Theorem~4.6(2)]{achar}, so they at least lie in ${}^p\Db\Cohmix(\cN)^{\le 0}$.

On the other hand, the third term of~\eqref{eq:ses-1} is concentrated in cohomological degree $3$ and has $0$-dimensional support, so by Remark~\ref{rmk:pcoh-0}, it is perverse-coherent.  It follows that the first term at least belongs to ${}^p\Db\Cohmix(\cN)^{\ge 0}$.  Combining this with the previous paragraph, we see that all three terms in~\eqref{eq:ses-1} are perverse-coherent.  In other words,~\eqref{eq:ses-1} is actually a short exact sequence in $\PCoh^\Gm(\cN)$.

Let $\cG_a = \tau^{\le 2}(\pi_{\cV*}\cO_\cV(-\alpha_0-\lambda_a)[-1])$.  We wish to prove that
\[
\cG_a \cong \IC([2,1],\bk_{-a})\la -3a-2\ra.
\]
To prove this, we must check the following conditions: 
\begin{enumerate}
\item $\cG_a$ is supported on $\ol{C_{[2,1]}}$, and $\cG_a|_{C_{[2,1]}}[1]$ is the line bundle $\bk_{-a}\la -3a-2\ra$.
\item $\cG_a $ has no simple subobject or quotient supported on $C_{[1,1,1]}$.
\end{enumerate}
That $\cG_a$ is supported on $\ol{C_{[2,1]}}$ is clear.  Since the third term of~\eqref{eq:ses-1} is supported (at least set-theoretically) on $C_{[1,1,1]}$, that distinguished triangle shows us that
\[
\cG_a|_{C_{[2,1]}}[1] \cong (\pi_{\cV*}\cO_{\cV}(-\alpha_0-\lambda_a))|_{C_{[2,1]}}.
\]
Since $-\alpha_0-\lambda_a = (-a-1,-a,2a+1)$, Lemma~\ref{lem:wt-restrict} tells us that this line bundle is indeed $\bk_{-a}\la-3a-2\ra$.  Thus, condition~(1) holds.

Recall from Remark~\ref{rmk:pcoh-0} that a perverse-coherent sheaf supported (set-theo\-ret\-ic\-ally) on $C_{[1,1,1]}$ is of the form $\cF[-3]$, where $\cF \in \Cohmix(\cN)$ has finite-dimensional global sections.  If $\cG_a$ had such an object as a quotient, we would have
\begin{equation*}
\Hom_{\Db\Cohmix(\cN)}(\tau^{\le 2} (\pi_{\cV*}\cO_{\cV}(-\alpha_0-\lambda_a)[-1]), \cF[-3]) \ne 0,
\end{equation*}
but this is impossible since $\cF[-3] \in \Db\Cohmix(\cN)^{\ge 3}$.

Suppose now that $\cG_a$ had a subobject $\cF[-3]$ supported on $C_{[1,1,1]}$.  Then this would also be a subobject of $\pi_{\cV*}\cO_{\cV}(-\alpha_0-\lambda_a)[-1]$.  Apply $\D$ to get a surjective map
\[
\D(\pi_{\cV*}\cO_{\cV}(-\alpha_0-\lambda_a)[-1]) \twoheadrightarrow \D(\cF[-3]).
\]
Now, $\D(\cF[-3])$ can be written in the form $\cF'[-3]$ for some $\cF' \in \Cohmix(\cN)$ that is again supported on $C_{[1,1,1]}$.  By Lemma~\ref{lem:V-properties}, the map above can be rewritten as a nonzero map $\pi_{\cV*}\cO_{\cV}(\lambda_a)\la 4\ra[-1] \twoheadrightarrow \cF'[-3]$. This map is a nonzero element of 
\begin{multline*}
\Hom_{\Db\Cohmix(\cN)}(\pi_{\cV*}\cO_{\cV}(\lambda_a)\la 4\ra[-1], \cF'[-3]) \\
\cong
\Ext^{-2}_{\Cohmix(\cN)}(\pi_{\cV*}\cO_{\cV}(\lambda_a)\la 4\ra, \cF'),
\end{multline*}
but this is nonsensical. (Here, we are using the fact from Lemma~\ref{lem:piv-cohom} that $\pi_{\cV*}\cO_{\cV}(\lambda_a)$ is a coherent sheaf.)
\end{proof}

The main result of this paper is the following.  In this statement, we normalize the grading shifts in accordance with Theorem~\ref{thm:graded-lv}.  Via the Lusztig--Vogan bijection (Table~\ref{tab:lv}), one can read off $\fC_\lambda$ for any $\lambda \in \bX^+$ from this theorem.

\begin{thm}
\phantomsection
\label{thm:main}
\begin{enumerate}
\item Let $\lambda \in \bX^+$, and let $L(\lambda)$ be the corresponding irreducible representation of $G$.  Let $i: C_{[1,1,1]} \hookrightarrow \cN$ be the inclusion map.  Then
\[
\IC([1,1,1], L(\lambda))\la 3\ra \cong i_*L(\lambda)[-3]\la 3\ra.
\]  
\item Let $a \ge 0$.  As graded $G$-modules, the cohomology sheaves of $\IC([2,1], \bk_a)\la 1\ra$ are given by
\[
\scalebox{0.70}{\hbox{$
\begin{array}{@{}c@{\ts}c@{\ts}c@{\ts}c@{\ts}c@{\ts}c@{\ts}c@{\ts}c@{\ts}c@{\ts}c@{\ts}c@{\ts}c@{}}
& -3a-1 & -3a+1 & \cdots & 3a-5 & 3a-3 & 3a-1 & 3a+1 & \cdots & 3a-3+2r & \cdots  \\
\cH^2: &H^1(\lambda_a - 3a\alpha_0) & H^1(\lambda_a - (3a-1)\alpha_0) & \cdots & H^1(\lambda_a-2\alpha_0) & 0 \\
\cH^1: &&&&& 0 &H^0(\lambda_a) & H^0(\lambda_a+\alpha_0) & \cdots & H^0(\lambda_a+r\alpha_0) & \cdots
\end{array}$}}
\]
The cohomology sheaves of $\IC([2,1],\bk_{-a})\la1\ra$ are given by
\[
\scalebox{0.70}{\hbox{$
\begin{array}{@{}c@{\ts}c@{\ts}c@{\ts}c@{\ts}c@{\ts}c@{\ts}c@{\ts}c@{\ts}c@{\ts}c@{\ts}c@{\ts}c@{}}
& -3a-1 & -3a+1 & \cdots & 3a-5 & 3a-3 & 3a-1 & 3a+1 & \cdots & 3a-3+2r & \cdots  \\
\cH^2: &H^1(\lambda_a^\sigma - 3a\alpha_0) & H^1(\lambda_a^\sigma - (3a-1)\alpha_0) & \cdots & H^1(\lambda_a^\sigma-2\alpha_0) & 0 \\
\cH^1: &&&&& 0 &H^0(\lambda_a^\sigma) & H^0(\lambda_a^\sigma+\alpha_0) & \cdots & H^0(\lambda_a^\sigma+r\alpha_0) & \cdots
\end{array}$}}
\]
\item We have $\IC([3], \bk) \cong \cO_\cN$.
\end{enumerate}
\end{thm}
\begin{proof}
The description of $\IC([1,1,1], L(\lambda))\la 3\ra$ is obvious (cf.~Remark~\ref{rmk:pcoh-0}). Next, the simple object $\IC([3],\bk) \cong \fC_{(0,0,0)}$ is isomorphic to $\Delta_{(0,0,0)} \cong \nabla_{(0,0,0)} \cong \pi_*\cO_\tcN$.  The latter is isomorphic to $\cO_\cN$ by, say,~\cite[Theorem~5.3.2]{bk}.

Next, according to Proposition~\ref{prop:ic-a-calc}, if $a \ge 0$, then $\IC([2,1],\bk_{-a})\la 1\ra$ is given by applying the shift $[-1]\la 3a+3\ra$ to the table from Lemma~\ref{lem:piv-cohom}, and then truncating.  The result is recorded above.

Finally, by Remark~\ref{rmk:ic-sigma}, we have $\IC([2,1],\bk_a)\la 1\ra \cong \IC([2,1], \bk_{-a})\la 1\ra^\sigma$.  The reasoning from the proof of Proposition~\ref{prop:pcoh-twist} shows that $(R\Ind_B^G \bk_\mu)^\sigma \cong R\Ind_B^G \bk_{-w_0\mu}$.  In particular,
\[
H^i(\lambda_a^\sigma + r\alpha_0)^\sigma \cong H^i(-w_0(\lambda_a^\sigma + r\alpha_0)) \cong H^i(\lambda_a + r\alpha_0).
\]
Thus, the description of $\IC([2,1],\bk_a)$ is obtained from that of $\IC([2,1], \bk_{-a})$ simply by replacing $\lambda_a^\sigma$ by $\lambda_a$ throughout.
\end{proof}

\begin{cor}\label{cor:main-0}
Suppose that $\bk$ has characteristic $0$.  If $a \ge 0$ is even, then $\IC([2,1],\bk_a)\la 1\ra$ is given by
\[
\scalebox{0.85}{\hbox{$
\begin{array}{@{}c@{\ts}c@{\ts}c@{\ts}c@{\ts}c@{\ts}c@{\ts}c@{\ts}c@{\ts}c@{\ts}c@{\ts}c@{\ts}c@{}}
& -3 & -1 & \cdots & 3a-5 & 3a-3 & 3a-1 & 3a+1 & \cdots & 3a -3 + 2r & \cdots \\
\cH^2: & 0 & H^1(\lambda_a - \frac{3a}{2}\alpha_0) & \cdots & H^1(\lambda_a-2\alpha_0) & 0 \\
\cH^1: & 0 &&&& 0 &H^0(\lambda_a) & H^0(\lambda_a+\alpha_0) & \cdots & H^0(\lambda_a + r\alpha_0) & \cdots
\end{array}$}}
\]
If $a \ge 1$ is odd, then $\IC([2,1],\bk_a)\la 1\ra$ is given by
\[
\scalebox{0.85}{\hbox{$
\begin{array}{@{}c@{\ts}c@{\ts}c@{\ts}c@{\ts}c@{\ts}c@{\ts}c@{\ts}c@{\ts}c@{\ts}c@{\ts}c@{}}
& -2 & \cdots & 3a-5 & 3a-3 & 3a-1 & 3a+3 & \cdots & 3a-3+2r & \cdots  \\
\cH^2:  & H^1(\lambda_a - \frac{3a+1}{2}\alpha_0) & \cdots & H^1(\lambda_a-2\alpha_0) & 0 \\
\cH^1: &&&& 0 &H^0(\lambda_a) & H^0(\lambda_a+\alpha_0) & \cdots & H^0(\lambda_a + r\alpha_0) & \cdots
\end{array}$}}
\]
\end{cor}
Of course, there are similar formulas for $\IC([2,1],\bk_{-a})\la 1\ra$ as well.
\begin{proof}
In the proof of Theorem~\ref{thm:main}, use the tables from Remark~\ref{rmk:piv-0} instead of those from Lemma~\ref{lem:piv-cohom}.
\end{proof}

\begin{rmk}
When $p > 0$, the cohomology modules $H^1(\lambda_a - r\alpha_0)$ are quite complicated, and their structure is still not completely understood. 
Their vanishing behavior differs from the $p=0$ case, and is instead determined by
 \cite[Theorem 4.5(i)]{and1979}. The characters of these objects are also complicated, but can actually be recursively computed by the formulas 
 appearing in \cite{hardesty}.\footnote{The recursive formulas originally appeared in \cite{donkin}, but were later found to contain mistakes by the second author of this paper.}
Furthermore, these recursive character formulas can be applied to obtain the ``characters'' (cf.~\cite[\S 4.2]{achar})  of the cohomology sheaves. 
\end{rmk}

\section{Applications and consequences}
\label{sec:appl}

\subsection{Cohomology of tilting modules for quantum groups}

Let $\zeta \in \C$ be a primitive $\ell$-th root of unity with $\ell > 3$ and odd. Also, let 
$\uenv_{\zeta} = \uenv_{\zeta}(\mathfrak{sl}_3(\C))$ be the Lusztig quantum group of ``adjoint type'' (see, for instance,~\cite[Remark~2.6.3]{abg}), and let $\su_{\zeta} = \su_{\zeta}(\mathfrak{sl}_{3}(\C))$ be the corresponding small quantum group.  Let $\Rep(\uenv_\zeta)$ be the category of finite-dimensional $\uenv_\zeta$-modules of type $1$.  For each $\mu \in \bX^+$, let $T_\zeta(\mu) \in \Rep(\uenv_\zeta)$ be the tilting module for $\uenv_\zeta$ of highest weight $\mu$.

Let $W_\aff = W \ltimes \bX$ be the affine Weyl group.\footnote{Since we are working in an adjoint group, $\bX$ is the root lattice, and $W \ltimes \bX$ is indeed a Coxeter group.} For $w = v \ltimes \lambda \in W_\aff$ and $\mu \in \bX$, we consider the ``dot action'' $w \bullet \mu = v(\mu + \ell\lambda + \rho) - \rho$, where, as usual, $\rho$ is one-half the sum of the positive roots.  (In our case, $\rho = \alpha_0 = (1,0,-1)$.)  For $\lambda \in \bX$, let $w_\lambda$ be the unique element of minimal length in the coset $W\lambda \subset W_\aff$.  By the linkage principle, a tilting module $T_\zeta(\mu)$ belongs to the block containing the trivial representation if and only if $\mu$ is of the form $w_\lambda \bullet 0$ for some $\lambda \in \bX$.

The main result of~\cite{bezru} computes the cohomology of tilting modules $T_\zeta(w_\lambda \bullet 0)$ in the principal block.  Specifically, according to~\cite[Theorem~1 and Corollary~2]{bezru}, there is an isomorphism of $G$-modules
\[
\Ext^\bullet_{\su_\zeta}(\C,T_\zeta(w_\lambda \bullet 0)) \cong 
\begin{cases}
R\Gamma(\fC_{w_0\lambda}) & \text{if $\lambda \in -\bX^+$,} \\
0 & \text{otherwise.}
\end{cases}
\]
(Note that in~\cite{bezru}, the simple perverse-coherent sheaves are labelled by \emph{antidominant} weights. The right-hand side above involves $w_0\lambda$ in order to match the conventions of the present paper.) The proof of~\cite[Proposition~9]{bezru} makes explicit the interaction of the gradings on either side: for $\lambda \in -\bX^+$, we have
\begin{equation}\label{eqn:t-cohom}
\Ext^k_{\su_\zeta}(\C, T_\zeta(w_\lambda \bullet 0)) = \bigoplus_{i \in \Z} \cH^i(\fC_{w_0\lambda})_{k-i}.
\end{equation}
Corollary~\ref{cor:main-0} lets us write down the right-hand side of this equation explicitly, and leads to the following observation.

\begin{prop}
Let $T \in \Rep(\uenv_\zeta)$ be a nontrivial indecomposable tilting module.  For any $k \ge 0$, the space $\Ext^k_{\su_\zeta}(\C, T)$ is either $0$ or an irreducible $PGL_3$-representation.
\end{prop}
\begin{proof}
If $T$ is a nontrivial tilting module with nonzero cohomology, then it is of the form $T(w_\lambda \bullet 0)$ with $\lambda \in -\bX^+$ and $\lambda \ne 0$.  Thus, $\fC_{w_0\lambda}$ is a simple perverse-coherent sheaf whose support is either $C_{[1,1,1]}$ or $\ol{C_{[2,1]}}$.  In the former case, the right-hand side of~\eqref{eqn:t-cohom} consists of a single term (for $i = 3$), and that term is irreducible.  In the latter case, the formulas in Corollary~\ref{cor:main-0} show us again that the right-hand side of~\eqref{eqn:t-cohom} has at most one nonzero term.  By the Borel--Weil--Bott theorem, that term is irreducible.
\end{proof}

\subsection{Failure of positivity}

When $\bk$ has characteristic $0$, it can be shown for any $G$ that perverse-coherent sheaves obey a strong $\Ext$-vanishing property: for all $\lambda,\mu \in \bX^+$, we have
\begin{equation}\label{eqn:mixed}
\Ext^1(\fC_\lambda, \fC_\mu\la n\ra) = 0 \qquad \text{if $n \ge 0$.}
\end{equation}
(This can be proved by converting the problem to one about mixed $\ell$-adic constructible sheaves on the dual affine flag variety, using the main result of~\cite{bezru-psaf}.  For a similar argument, see the proof of~\cite[Lemma~9]{bezru}.)  This condition means that $\PCoh^\Gm(\cN)$ is a ``mixed category'' in the sense of~\cite[Definition~4.1.1]{bgs}.

In positive characteristic, condition~\eqref{eqn:mixed} is always false: because the representation theory of $G$ is not semisimple, one can always find counterexamples to~\eqref{eqn:mixed} with $n = 0$, and where $\fC_\lambda$ and $\fC_\mu$ are both supported on the zero nilpotent orbit.

However, the calculations in~\cite{achar} show that for $G = SL_2$, this is the extent of the failure: one still has
\begin{equation}\label{eqn:positive}
\Ext^1(\fC_\lambda,\fC_\mu\la n\ra) = 0 \qquad \text{if $n > 0$.}
\end{equation}
If this condition holds, we say that $\PCoh^\Gm(\cN)$ is \emph{positively graded}, because it is analogous to the category of graded modules over a positively graded ring.  In particular,~\cite[Lemma~4.1.2]{bgs} still holds: every object is equipped with a canonical filtration, and every morphism is strictly compatible with these filtrations.  For a discussion of positively graded categories in the context of parity sheaves, see~\cite{ar:mpsfv3}.

\begin{prop}
For $G = PGL_3$, the category $\PCoh^\Gm(\cN)$ is positively graded if and only if $\bk$ has characteristic $0$.
\end{prop}
\begin{proof}
The ``if'' direction has been discussed above in the context of~\eqref{eqn:mixed}.  Now suppose that $\bk$ has positive characteristic (different from $2$ and $3$, so that (H1)--(H3) hold).  We must exhibit a counterexample to~\eqref{eqn:positive}.

We first claim that there exists a positive integer $a > 0$ and an integer $n$ with $0 < n < \frac{3}{2}a$ such that the representation
\begin{equation}\label{eqn:nz-cohom}
H^1(\lambda_a - (3a + 1 -n)\alpha_0) = H^1(n-1-2a,a,-n+1+a)
\end{equation}
is nonzero.  For instance, we may take $a = p$ and $n = p+1$, where we note that since $p>2$, we have $n=p+1 < 3p/2$. In this case, 
\[
(n-1-2a,a,-n+1+a) = (-p, p, 0), 
\]
and the nonvanishing of~\eqref{eqn:nz-cohom} follows from \cite[Theorem~4.5(i)]{and1979}. 

Choose an integer $a > 0$ so that~\eqref{eqn:nz-cohom} is nonzero for some $n$, and then choose $n$ to be as small as possible.  The representation~\eqref{eqn:nz-cohom} is then the leftmost nonzero term in the description of $\cH^2(IC([2,1],\bk_a)\la 1\ra)$ from Theorem~\ref{thm:main}.  In graded module language, this is the lowest nonzero homogeneous component of the graded $\bk[\cN]$-module $\cH^2(IC([2,1],\bk_a)\la 1\ra)$; it occurs in grading degree $2n-3a-3$.  Choose an irreducible quotient $L$ of $H^1(n-1-2a,a,-n+1+a)$. We deduce that there is a surjective map
\[
\cH^2(IC([2,1],\bk_a)\la 1\ra) \to i_*L\la 3a+3-2n\ra,
\]
where $i: C_{[1,1,1]} \hookrightarrow \cN$ is the inclusion map.  Next, note that $\cH^2(IC([2,1],\bk_a)\la 1\ra) \cong (\tau^{\ge 2}\IC([2,1],\bk_a)\la 1\ra)[2]$. It follows from the adjunction properties of truncation that there is a nonzero map
\[
\IC([2,1],\bk_a)\la 1\ra \to i_*L[-2]\la 3a+3-2n\ra.
\]
This map is an element of
\begin{multline*}
\Hom(\IC([2,1],\bk_a)\la 1\ra, i_*L[-2]\la 3a+3-2n\ra) \\
\cong \Ext^1_{\PCoh^\Gm(\cN)}(\IC([2,1],\bk_a)\la 1\ra, (i_*L[-3]\la3\ra) \la 3a - 2n \ra).
\end{multline*}
(Here we use~\cite[Remarque~3.1.17(ii)]{bbd} to convert the $\Hom$-group to an $\Ext^1$-group.)  By assumption, we have $3a -2n > 0$, so this contradicts~\eqref{eqn:positive}.
\end{proof}

\end{document}